\documentclass[11pt,a4paper]{amsart}
\usepackage[utf8]{inputenc}
\usepackage{amssymb}
\usepackage{amsmath}
\usepackage{amsthm}
\usepackage[a4paper,margin=1in]{geometry}
\numberwithin{equation}{section}
\usepackage{hyperref}
\usepackage{pbox}
\usepackage{graphicx}
\usepackage{amsrefs}

\DefineSimpleKey{bib}{primaryclass}{}
\DefineSimpleKey{bib}{archiveprefix}{}

\BibSpec{arXiv}{%
  +{}{\PrintAuthors}{author}
  +{,}{ \textit}{title}
  +{}{ \parenthesize}{date}
  +{,}{ arXiv }{eprint}
  +{,}{ primary class }{primaryclass}
}

\usepackage{caption}
\usepackage{subcaption}
\usepackage{float}
\usepackage{pax}
\usepackage{pdfpages}
\usepackage{enumerate}
\usepackage[shortlabels]{enumitem}

\usepackage{tikz}
\usepackage{tikz-cd}
\usepackage{mathtools}
\usepackage{comment}
\usepackage{blkarray}

\theoremstyle{plain}
\newtheorem{theorem}{Theorem}[section]

\theoremstyle{definition}
\newtheorem{remark}[theorem]{Remark}

\theoremstyle{definition}
\newtheorem{definition}[theorem]{Definition}
\newtheorem{exmp}[theorem]{Example}

\theoremstyle{plain}

\theoremstyle{plain}
\newtheorem{lemma}[theorem]{Lemma}

\theoremstyle{plain}

\DeclareMathOperator{\reg}{reg}

\DeclareMathOperator{\dimn}{dim}

\DeclareMathOperator{\hilb}{Hilb}

\DeclareMathOperator{\bigin}{bigin}
\DeclareMathOperator{\multigin}{multigin}

\DeclareMathOperator{\pic}{Pic}

\DeclareMathOperator{\maxdeg}{maxdeg}
\DeclareMathOperator{\cox}{Cox}
\DeclareMathOperator{\nef}{Nef}

\title{Gotzmann's persistence theorem for smooth projective toric varieties}
\author{Patience Ablett}
\address{Warwick Mathematics Institute, University of Warwick, Coventry, CV4 7EZ}
\email{patience.ablett@warwick.ac.uk}

\usepackage{calc}
\begin{document}

\maketitle
\begin{abstract}
    Gotzmann's persistence theorem enables us to confirm the Hilbert polynomial of a subscheme of projective space by checking the Hilbert function in just two points, regardless of the dimension of the ambient space. We generalise this result to products of projective spaces, and then extend our result to any smooth projective toric variety. The number of points to check depends solely on the Picard rank of the ambient space, with no dependence on the dimension.
\end{abstract}

\section{Introduction}

The Hilbert scheme $\hilb_P(\mathbb{P}^n)$ is a widely studied object in algebraic geometry. From an algebraic perspective this scheme parameterises homogeneous saturated ideals with a given Hilbert polynomial. In~\cite{haiman2002multigraded} Haiman and Sturmfels extended these ideas to the multigraded setting, where the Hilbert scheme parameterises homogeneous ideals with a given Hilbert function in a polynomial ring graded by some abelian group. A case of particular interest is when the multigraded ring in question is the Cox ring, denoted $\cox(X)$, of a smooth projective toric variety $X$. In this case the Hilbert function of a homogeneous ideal in $\cox(X)$ eventually agrees with a polynomial for degrees sufficiently far into the nef cone of $X$. We can therefore define $\hilb_P(X)$ in an analogous manner to $\hilb_P(\mathbb{P}^n)$. We consider the parameter space of homogeneous ideals in $\cox(X)$ which are saturated with respect to the irrelevant ideal of $X$ and have Hilbert polynomial $P$.

It is natural to ask what properties of standard-graded Hilbert schemes have an extension to the smooth projective toric variety case. Of particular interest are two theorems of Gotzmann, regularity and persistence, which are used in the explicit construction of the Hilbert scheme. Gotzmann's regularity theorem gives a bound on the Castelnuovo-Mumford regularity of a standard-graded ideal as introduced in~\cite{mumford1966lectures}. Gotzmann's persistence theorem can be informally restated in the following way. For a homogeneous ideal $I$ in a standard-graded polynomial ring, a Hilbert polynomial $P(t)$, and sufficiently large $d \in \mathbb{N}$, checking that $H_I(d)=P(d)$ and $H_I(d+1)=P(d+1)$ guarantees that $P_I(t)=P(t)$. Here, $H_I(d)$ denotes the Hilbert function of $S/I$ and $P_I(t)$ denotes the associated Hilbert polynomial. For a formal statement, see Theorem \ref{persistence}. The surprising aspect here is that by checking the value of $H_I(d)$ in just two points we have identified the polynomial $P_I(t)$, as opposed to the expected $\deg(P_I)+1$ points. Combining this result with Gotzmann's regularity theorem allows us to obtain explicit equations for the Hilbert scheme $\hilb_P(\mathbb{P}^n)$.

Maclagan and Smith define Castelnuovo-Mumford regularity for the multigraded case~\cite{maclagan2004multigraded} and generalise Gotzmann's regularity theorem~\cite{maclagan2003uniform} to any smooth projective toric variety. Their generalisation recovers Gotzmann's original result when $X = \mathbb{P}^n$. In this paper we generalise our earlier informal restatement of Gotzmann's persistence theorem to any smooth projective toric variety. We begin with the case of the product of projective spaces, before extending to the more general setting.

We introduce some terminology that will allow us to state Theorem \ref{maintheorem1}. Let $X$ be a smooth projective toric variety of Picard rank $s$. For a homogeneous ideal $J \subseteq \cox(X)$, Definition \ref{whatispoly} gives a formal description of the associated polynomial $P_J(t_1,\dots,t_s)$. We say that $P(t_1,\dots,t_s) \in \mathbb{Q}[t_1,\dots,t_s]$ is a Hilbert polynomial on $\cox(X)$ if there exists some ideal $J \subseteq \cox(X)$, homogeneous with respect to the $\mathbb{Z}^s$-grading, such that $P_J=P$.
\begin{theorem}\label{maintheorem1}
    Let $S=\cox(\mathbb{P}^{n_1} \times \dots \times \mathbb{P}^{n_s})$, let $I \subseteq S$ be an ideal, homogeneous with respect to the $\mathbb{Z}^s$-grading, and let $P(t_1,\dots,t_s) \in \mathbb{Q}[t_1,\dots,t_s]$ be a Hilbert polynomial on $S$. Then there exists a point $(d_1, \dots, d_s) \in \mathbb{N}^s$ such that if $H_I(b_1,\dots,b_s) = P(b_1, \dots, b_s)$ for all points $(b_1,\dots,b_s) \in \mathbb{N}^s$ with $b_i \in \{d_i,d_i+1\}$, then $P_I=P$.
\end{theorem}
Theorem \ref{maintheorem1} says that for an ideal $I$ in the Cox ring of the product of $s$ projective spaces we can confirm $P_I(t_1,\dots,t_s)=P(t_1,\dots,t_s)$ simply by checking $H_I(b_1,\dots,b_s)$ at the vertices of an $s$-dimensional hypercube in $\mathbb{N}^s$. This is therefore a natural $s$-dimensional generalisation of our earlier informal restatement of Gotzmann's original persistence result. Note that the proof of Theorem \ref{maintheorem1} is constructive, allowing us to find appropriate $(d_1,\dots,d_s)$ for a given $I$ and $P$. As outlined in \cite{maclagan2003uniform}*{Algorithm 6.3}, $P_I(t_1,\dots,t_s)$ can have as many as $\binom{n_1+ \dots + n_s + s}{s}$ coefficients, so we would naively expect to check $H_I(d_1,\dots,d_s)$ in this many points to verify $P_I(t_1,\dots,t_s)$. Even when setting $n_i=1$ for all $i$ we have $\binom{n_1+ \dots + n_s + s}{s} \geq 2^s$ for all $s \in \mathbb{Z}_{\geq 1}$. Further, $\binom{n_1+ \dots + n_s + s}{s}$ grows significantly faster than $2^s$ as $s$ increases, meaning our theorem is a significant improvement.

In the case of $\mathbb{P}^n \times \mathbb{P}^m$ we contrast our result with Crona's persistence-type result \cite{crona2006standard}*{Theorem 4.10} for products of two projective spaces. Crona's theorem states that if $H_I(b_1+1,b_2)$ displays maximal growth from $H_I(b_1,b_2)$, then $H_I(u,b_2)$ continues to display maximal growth for all $u \geq b_1$. Crona's result predicts the behaviour of the Hilbert function $H_I(b_1,b_2)$ as $b_1$ grows for \textit{fixed} values of $b_2$, and vice versa. However, our result can predict growth as \textit{both} $b_1$ and $b_2$ increase in value, meaning we can satisfy criterion (h) of \cite{haiman2002multigraded}*{Proposition 3.2} for a supportive set. We expand on this difference in Example \ref{counter2}. Further work in this area includes that of Favacchio~\cite{lexsegment} which characterises the Hilbert functions of ideals in the Cox ring of $\mathbb{P}^n \times \mathbb{P}^m$, and work of Gasharov~\cite{gasharov1997extremal} which generalises Gotzmann's persistence theorem to finitely generated modules over a standard-graded polynomial ring.

The key new ideas in the proof of Theorem \ref{maintheorem1} are Lemma \ref{standecomp} and Lemma \ref{justbistable}. Lemma \ref{standecomp} allows us to write the Hilbert polynomial uniquely in terms of standard-graded Hilbert polynomials multiplied by binomial coefficients. This in turn enables us to use Gotzmann's original persistence theorem in the proof of Theorem \ref{maintheorem1}. Lemma \ref{justbistable} justifies an extension of results on bilex ideals to the multilex case.

We then extend Theorem \ref{maintheorem1} to more general smooth projective toric varieties. The situation is particularly nice for Picard-rank-2 varieties, whose Cox rings are described explicitly at the beginning of Section \ref{three}.
\begin{theorem}\label{picrank2}
    Let $J \subseteq R$, where $R$ is the Cox ring of a smooth Picard-rank-2 toric variety and $J$ is an ideal homogeneous with respect to the $\mathbb{Z}^2$-grading. Let $P(t_1,t_2)$ be a Hilbert polynomial on $R$. Then there exists $(d_1,d_2) \in \mathbb{N}^2$ such that if \begin{gather*}
        H_J(d_1,d_2)=P(d_1,d_2), \quad H_J(d_1+1,d_2)=P(d_1+1,d_2), \\ H_J(d_1,d_2+1)=P(d_1,d_2+1), \quad H_J(d_1+1,d_2+1)=P(d_1+1,d_2+1), 
    \end{gather*}
    then $P_J=P$.
\end{theorem}
For varieties with higher Picard rank, we obtain the following theorem. In particular, our result still depends solely on the Picard rank of the variety.
\begin{theorem}\label{pichighpers}
Let $X$ be a smooth projective toric variety of Picard rank $s$. Let $R=\cox(X)$, with a $\mathbb{Z}^s$-grading resulting from the identification $\pic(X) \cong \mathbb{Z}^s$. Let $J \subseteq R$ be an ideal, homogeneous with respect to the $\mathbb{Z}^s$-grading. Let $P(t_1,\dots,t_s)$ be a Hilbert polynomial on $R$. Then there exists at most $2^s$ points $(r_1,\dots,r_s) \in \mathbb{N}^{s}$ such that checking $H_J(r_1,\dots,r_s)=P(r_1,\dots,r_s)$ for all of these points guarantees that $P_J=P$. 
\end{theorem}
The $2^{s}$ points of Theorem \ref{pichighpers} form a zonotope, generalising the hypercube seen in Theorems \ref{maintheorem1} and \ref{picrank2}. Again, the Hilbert polynomial of a subscheme of a $d$-dimensional smooth projective toric variety $X$ of Picard rank $s$ can have up to $\binom{s+d}{d}$ coefficients. This means that naively we have to check $\binom{s+d}{d}$ points to find this polynomial. Theorems \ref{picrank2} and \ref{pichighpers} ensure that we only have to check $2^s$ points to find the Hilbert polynomial for any smooth projective toric variety. For $d \geq s$, we have $\binom{s+d}{d} \geq 2^s$ for all $s \geq 1$, and the binomial coefficient also grows faster as $s$ increases. The key contribution is that the complexity of finding the Hilbert polynomial no longer depends on the dimension of $X$, and only depends on the Picard rank.

The rest of this paper is laid out as follows. In Section \ref{two} we outline necessary background on strongly multistable ideals, allowing us to prove results relating to the product of $s$ projective spaces in Section \ref{twohalf}. We conclude Section \ref{twohalf} with a proof of Theorem \ref{maintheorem1} and related remarks. In Section \ref{three} we see how Theorem \ref{maintheorem1} can be applied to smooth projective toric varieties of Picard rank $s \geq 2$, where we prove Theorems \ref{picrank2} and \ref{pichighpers}.

\section*{Acknowledgements}
Many thanks to Diane Maclagan for her continued support throughout this project. The author was funded through the Warwick Mathematics Institute Centre for
Doctoral Training, with support from the University of Warwick and the UK Engineering and Physical Sciences Research Council (EPSRC grant EP/W523793/1).

\section{Strongly multistable ideals}\label{two}
 \subsection{Gotzmann's persistence theorem}

In this section we outline some necessary background to prove Theorem \ref{maintheorem1}. We begin by stating Gotzmann's original persistence theorem. To do so we require the following lemma and definition.
\begin{lemma}[\cite{bruns1998cohen}*{Lemma 4.2.6}]
Fix $d \in \mathbb{Z}_{>0}$ and $\alpha \in \mathbb{N}$. Then $\alpha$ can be written uniquely in the form
\begin{equation*}
    \alpha=\binom{\kappa(d)}{d}+\binom{\kappa(d-1)}{d-1}+\dots+\binom{\kappa(1)}{1},
\end{equation*}
where $\kappa(d)>\kappa(d-1)>\dots>\kappa(1)\geq0$.
\end{lemma}
\begin{definition}
Let $\alpha \in \mathbb{N}$, $d \in \mathbb{Z}_{>0}$. Define
\begin{equation*}
    \alpha^{\left<d\right>}= \binom{\kappa(d)+1}{d+1}+\binom{\kappa(d-1)+1}{d-1+1}+\dots +\binom{\kappa(1)+1}{1+1},
\end{equation*}
with $\kappa(d),\dots,\kappa(1)$ as above. We set $0^{\left<d\right>}=0$.
\end{definition}
A key result of Macaulay~\cite{macaulay1927some} shows that for $b \in \mathbb{Z}_{>0}$ we have
\[H_I(b+1) \leq H_I(b)^{\left<b\right>}.\]
We now state Gotzmann's persistence theorem.
\begin{theorem}[\cite{gotzmann1978bedingung}, see also~\cite{bruns1998cohen} Theorem 4.3.3]\label{persistence}
    Let $S=k[x_0,\dots,x_n]$ be a standard-graded polynomial ring, and let $I \subseteq S$ be a homogeneous ideal generated in degrees $\leq d$. Suppose $H_I(d+1) = H_I(d)^{\left<d\right>}$. Then $H_I(b+1)=H_I(b)^{\left< b \right>}$ for all $b \geq d$.
\end{theorem}
As highlighted in the introduction, we can relate this theorem to the Hilbert polynomial of $I$. In general, any Hilbert polynomial $P(t) \in \mathbb{Q}[t]$ of a standard-graded ideal can be written in the form
\begin{equation*}
    P(t) = \sum_{i=1}^{D} \binom{t+\alpha_i-i+1}{t-i+1}
\end{equation*}
for some $D \in \mathbb{Z}_{>0}$ and $\alpha_1 \geq \alpha_2 \geq \dots \geq \alpha_D \geq 0$. This number $D$ is called the Gotzmann number of $P(t)$. For $t \geq D-1$, we observe that $P(t+1)=P(t)^{\left<t\right>}$, which we expect from Theorem \ref{persistence}. 
\subsection{Strongly multistable ideals}\label{twotwo}
We now outline some necessary background on strongly multistable ideals. Throughout the rest of Sections \ref{two} and \ref{twohalf} let \[S = k[x_{1,0},\dots,x_{1,n_1},\dots,x_{s,0},\dots,x_{s,n_s}]\] be the Cox ring of the product of $s$ projective spaces $\mathbb{P}^{n_1} \times \dots \times \mathbb{P}^{n_s}$, with $\deg(x_{i,j})=\boldsymbol{e}_i \in \mathbb{Z}^s$. We give $S$ the degree lexicographic monomial ordering induced by $x_{1,0} \succ \dots \succ x_{1,n_1} \succ \dots \succ x_{s,0} \succ \dots \succ x_{s,n_s}$. More succinctly, we have $x_{i,j} \succ x_{k,l}$ if $i<k$, or $i=k$ and $j<l$. We use $\boldsymbol{x^u}$ as shorthand for a monomial in $S$, with $\boldsymbol{x^u}=\prod x_{i,j}^{u_{i,j}}$ for some $u_{i,j} \in \mathbb{N}$.

We also set the following piece of notation. An ideal $I$ is generated in degrees $\leq (a_1,\dots,a_s) \in \mathbb{N}^s$ if there exist homogeneous generators $\{f_1,\dots,f_r\}$ for $I$ such that for each $(d_{i1},\dots,d_{is}) = \deg(f_i)$, we have $d_{ij} \leq a_j$.
\begin{definition}\label{multistab}
    A monomial ideal $I \subseteq S$ is strongly multistable if for all monomials $\boldsymbol{x^u} \in I$, $1 \leq i \leq s$ and $0 \leq j \leq n_i$, we have that if $x_{i,j} \mid \boldsymbol{x^u}$ then $x_{i,k} \tfrac{\boldsymbol{x^u}}{x_{i,j}} \in I$ for all $0 \leq k < j$.
\end{definition}

\begin{definition}
    For a monomial $\boldsymbol{x^u} \in S$ define
    \begin{align*}
    m^i(\boldsymbol{x^u}) & = \max\{j \colon x_{i,j} \mid \boldsymbol{x^u}\},\\ 
    m_i(\boldsymbol{x^u}) & = \min\{j \colon x_{i,j} \mid \boldsymbol{x^u}\}.
    \end{align*}
\end{definition}

For every homogeneous ideal of $S$ there exists a strongly multistable ideal with the same Hilbert function. One key example is the multigeneric initial ideal, $\multigin(I)$ as outlined in~\cite{conca2022radical}. Work of Crona~\cite{crona2006standard} focuses on the specific case of strongly bistable ideals, which lie in the Cox ring of $\mathbb{P}^n \times \mathbb{P}^m$, which we denote $T$. In this case the bigeneric initial ideal was first defined in \cite{aramova2000bigeneric}.

Lemma 2.4 of~\cite{crona2006standard} leads to a partial decomposition of the quotient ring $T/I$ for a strongly bistable ideal $I \subseteq T$, with similarities to Stanley decompositions. This partial decomposition allows us to understand the Hilbert function of a strongly bistable ideal. We can generalise this lemma to the strongly multistable case, with the proof being a direct generalisation of Crona's proof.

\begin{lemma}\label{standec}
    Let $I \subseteq S$ be a strongly multistable ideal generated in degrees $\leq ( a_1, \dots, a_s)$. For a monomial $\boldsymbol{x^u} \in S_{(a_1,\dots,a_s)}$ and integers $\{v_i \geq a_i\}$ define
    \[H(\boldsymbol{x^u}) = \{\boldsymbol{x^u} \boldsymbol{x^v} \mid \boldsymbol{x^v} \in  S_{(v_1-a_1,\dots,v_s-a_s)}, m_i(\boldsymbol{x^v}) \geq m^i(\boldsymbol{x^u})\}.\]
    Then  $H = \bigcup_{\boldsymbol{x^u} \in I_{(a_1,\dots,a_s)}} H(\boldsymbol{x^u})$ generates the vector space $I_{(v_1,\dots,v_s)}$. Moreover $H(\boldsymbol{x^u}) \cap H(\boldsymbol{x^{u'}}) = \emptyset$ for $\boldsymbol{x^u} \neq \boldsymbol{x^{u'}}$, so $H$ is a monomial basis for $I_{(v_1,\dots,v_s)}$.
\end{lemma}
\begin{proof}
    To show that $H$ generates $I_{(v_1,\dots,v_s)}$ as a $k$-vector space, we show that for any monomial \[\boldsymbol{x^u} \in I_{(a_1,\dots,a_s)}\] and monomial $\boldsymbol{x^v} \in S_{(a_1-v_1,\dots,a_s-v_s)}$, then $\boldsymbol{x^u} \boldsymbol{x^v}$ lies in $H$. Suppose $m_i(\boldsymbol{x^v}) < m^i(\boldsymbol{x^u})$ for some $i$. It follows from the strongly multistable property that \[\boldsymbol{x}^{\boldsymbol{u}_{2}}= x_{i,m_i(\boldsymbol{x^v})} \tfrac{\boldsymbol{x^u}}{x_{i,m^i(\boldsymbol{x^u})}} \in I_{(a_1,\dots,a_s)}.\] Further if we let \[\boldsymbol{x}^{\boldsymbol{v}_{2}}=x_{i,m^i(\boldsymbol{x^u})}\tfrac{\boldsymbol{x^v}}{x_{i,m_i(\boldsymbol{x^v})}}\] then $\boldsymbol{x^u}\boldsymbol{x^v}=\boldsymbol{x}^{\boldsymbol{u}_{2}}\boldsymbol{x}^{\boldsymbol{v}_{2}}$, $m_i(\boldsymbol{x}^{\boldsymbol{v}_{2}}) \geq m_i(\boldsymbol{x^v})$ and $m^i(\boldsymbol{x}^{\boldsymbol{u}_{2}}) \leq m^i(\boldsymbol{x^u})$. We now repeat the above process for $\boldsymbol{x}^{\boldsymbol{u}_{2}}$ and $\boldsymbol{x}^{\boldsymbol{v}_{2}}$. If $k$ is the maximum power of $x_{i,{m_i({\boldsymbol{x^v}}})}$ dividing $\boldsymbol{x^v}$ then $k$ repeats is enough to ensure that $m_i(\boldsymbol{x}^{\boldsymbol{v}_{k}}) >m_i(\boldsymbol{x^{v}})$. It follows that after a finite number $n$ of repeats $m_i(\boldsymbol{x}^{\boldsymbol{v}_{n}}) \geq m^i(\boldsymbol{x}^{\boldsymbol{u}_{n}})$. Thus we can assume $m_i(\boldsymbol{x^v}) \geq m^i(\boldsymbol{x^u})$ for all $i$, completing the first part of the proof. To see that $H(\boldsymbol{x^u}) \cap H(\boldsymbol{x^{u'}}) = \emptyset$ for $\boldsymbol{x^u} \neq \boldsymbol{x^{u'}}$, let \[\boldsymbol{x^u}=c_1 \dots c_s, \quad \boldsymbol{x^{u'}}=c_1' \dots c_s'\] with $c_i, c_i'$ monomials in $S_{a_i \boldsymbol{e}_i}$. Suppose $\boldsymbol{x^u} \neq \boldsymbol{x^{u'}}$. Then without loss of generality there exists a pair $\{c_i,c_i'\}$ with $c_i \succ c_i'$. Let $j$ be the smallest index such that the power of $x_{i,j}$ is larger for $c_i$ than $c_i'$. Note that $m^i(\boldsymbol{x^{u'}})=m^i(c_i')>j$ since $c_i$ and $c_i'$ have equal degree. It therefore follows that for a monomial $\boldsymbol{x^v} \in H(\boldsymbol{x^u})$ the power of $x_{i,j}$ dividing $\boldsymbol{x^v}$ is strictly larger than for any monomial lying in $H(\boldsymbol{x^{u'}})$. Thus $H(\boldsymbol{x^u}) \cap H(\boldsymbol{x^{u'}}) = \emptyset$.
\end{proof}
Note that \[\bigcup_{\boldsymbol{x^u} \in S_{(a_1,\dots,a_s)}} H(\boldsymbol{x^u})\] partitions the monomials of degree $(v_1,\dots,v_s)$. Further, \[\bigcup_{\boldsymbol{x^u} \in I_{(a_1,\dots,a_s)}} H(\boldsymbol{x^u})\] and 
\begin{equation}\label{parti}
\bigcup_{\boldsymbol{x^u} \in S_{(a_1,\dots,a_s)} \setminus I_{(a_1,\dots,a_s)}} H(\boldsymbol{x^u})
\end{equation}
partition the monomials forming a basis for $I_{(v_1,\dots,v_s)}$ and $(S/I)_{(v_1,\dots,v_s)}$ respectively. For $\sigma=\{c_1,\dots,n_1,c_2,\dots,n_2,\dots,c_s,\dots,n_s\}$ let \[S_{\sigma} = k[x_{i,j} \mid c_i \leq j \leq n_i].\] Then the partition \eqref{parti} shows that for \[\sigma_{\boldsymbol{u}} = \{m_1(\boldsymbol{x^u}), \dots, n_1, m_2(\boldsymbol{x^u}), \dots, n_2, \dots, m_s(\boldsymbol{x^u}), \dots, n_s\},\]
\[\bigoplus_{\boldsymbol{x^u} \in S_{(a_1,\dots,a_s)} \setminus I_{(a_1,\dots,a_s)}} \boldsymbol{x^u} S_{\sigma_{\boldsymbol{u}}}\] is a decomposition of $(S/I)_{(\geq a_1,\dots,\geq a_s)}$. This leads to the following definition.
\begin{definition}
    Let $I \subseteq S$ be a strongly multistable ideal generated in degrees $\leq (a_1, \dots, a_s)$. A partial Stanley decomposition of $S/I$ is a set \[\mathfrak{S}_{(a_1,\dots,a_s)}=\{(\boldsymbol{x^u},\sigma_{\boldsymbol{u}}) \mid \boldsymbol{x^u} \text{ is a monomial in } S_{(a_1,\dots,a_s)} \setminus I_{(a_1,\dots,a_s)}\}.\]
    We have 
    \[\dimn_k((S/I)_{(b_1,\dots,b_s)}) = \sum_{\mathfrak{S}_{(a_1,\dots,a_s)}} \dimn_k((S_{{\sigma}_{\boldsymbol{u}}})_{(b_1-a_1,\dots,b_s-a_s)}) \] for $(b_1,\dots,b_s) \in \mathbb{N}^s$ with $b_i \geq a_i$ for all $i$.
\end{definition}
This partial Stanley decomposition therefore allows us to calculate the Hilbert function of $S/I$ in sufficiently high degrees. By passing from an ideal $I$ to $\multigin(I)$ we can apply this idea to more general ideals.
\begin{exmp}
    Let $ T = k[x_0,x_1,y_0,y_1]$ be the Cox ring of $\mathbb{P}^1 \times \mathbb{P}^1$, and let $I=(x_0,x_1y_0) \subseteq T$ be an ideal. Since $I$ is strongly bistable and generated in degrees $\leq(2,2)$ we can construct a partition using the monomials of degree $(2,2)$. The only monomial in $T_{(2,2)}\setminus I_{(2,2)}$ is $\boldsymbol{x^u}=x_1^2y_1^2$. Here $m_x(\boldsymbol{x^u})=1=m_y(\boldsymbol{x^u})$ and $\sigma_{\boldsymbol{u}}=\{1,1\}$. Therefore $\mathfrak{S}_{(2,2)}=(x_1^2y_1^2,\{1,1\})$, and \[\dimn_k((T/I)_{(b_1,b_2)})=\dimn_k(k[x_1,y_1]_{(b_1-2,b_2-2)})\] for $b_1, b_2 \geq 2$.
\end{exmp}

\subsection{Multilex ideals}
Aramova, Crona and De Negri \cite{aramova2000bigeneric}*{Section 4} give a generalisation of lexicographic ideals for the bigraded case, called bilex ideals, and prove a number of useful results. As highlighted by Aramova, Crona and De Negri, these results generalise further to the multigraded case, with similar proofs. For completeness we will include a formal definition of multilex ideals, and a proof of the generalised version of one of their main theorems.
\begin{definition}\label{xlex}
    Let $M \subseteq S_{(a_1,\dots,a_s)}$ be a set of monomials. Then:
    \begin{enumerate}[\normalfont(i)]
        \item $M$ is $x_i$-lex if for all $\boldsymbol{x^u} \in S_{a_i\boldsymbol{e}_i}$, $\boldsymbol{x^v} \in S_{(a_1,\dots,a_{i-1},0,a_{i+1},\dots,a_s)}$ we have that $\boldsymbol{x^u} \boldsymbol{x^v} \in M$ implies that $\boldsymbol{x^{u'}} \boldsymbol{x^v} \in M$ for all $\boldsymbol{x^{u'}} \in S_{a_i\boldsymbol{e}_i}$ with $\boldsymbol{x^{u'}} \succ \boldsymbol{x^u}$.
        \item $M$ is multilex if it is $x_i$-lex for all $1 \leq i \leq s$. It is $x_1,\dots,x_i$-lex if it is $x_j$-lex for $1 \leq j \leq i$.
        \item A monomial ideal $I \subseteq S$ is multilex if $I_{(a_1,\dots,a_s)}$ is generated by a multilex set for all $(a_1,\dots,a_s) \in \mathbb{N}^s$.
        \end{enumerate}
\end{definition}
Observe that a multilex ideal is always strongly multistable. In the bigraded case, where $s=2$, Aramova Crona and De Negri \cite{aramova2000bigeneric}*{Lemma 4.13} show that for every strongly bistable ideal $I$, there is a multilex ideal $I^{\text{bilex}}$ with the same Hilbert function. We verify that this result can be extended to the strongly multistable case. To do so, we will need the following definitions.
\begin{definition}
    For a set of monomials $M \subseteq S_{(a_1,\dots,a_s)}$ we can decompose \[M = \bigcup_{j=1}^{k} M_j \boldsymbol{x}^{\boldsymbol{u}_j}\] for some $\boldsymbol{x}^{\boldsymbol{u}_j} \in S_{(a_1,\dots,a_{i-1},0,a_{i+1},\dots,a_s)}$ and $M_j \subseteq S_{a_i\boldsymbol{e}_i}$. Define \[M^{x_i \text{lex}} = \bigcup_{j=1}^{k} M_j^{\text{lex}} \boldsymbol{x}^{\boldsymbol{u}_j}.\] Let $M^{x_1 \dots x_i \text{lex}} = (((M^{x_1 \text{lex}})^{x_2 \text{lex}})^{\dots})^{x_i \text{lex}}$.
\end{definition}
\begin{remark}\label{subseteqs}
    For any sets of monomials $A \subseteq B \subseteq S_{(a_1,\dots,a_s)}$, and any $1 \leq i \leq s$, $A^{x_i \text{lex}} \subseteq B^{x_i \text{lex}}$.
\end{remark}

\begin{definition} Given a monomial $\boldsymbol{x^u} \in S_{(a_1,\dots,a_{s})}$, write $\boldsymbol{x^u}=\boldsymbol{x_1^{u_1}} \ldots \boldsymbol{x_s^{u_s}}$, with each $\boldsymbol{x_i^{u_i}} \in S_{a_i\boldsymbol{e}_i}$. Define the multi-lexsegment of $\boldsymbol{x^u}$ by 
    \[L(\boldsymbol{x^u}) = \{\boldsymbol{x_1^{u'_1}} \ldots \boldsymbol{x_{s}^{u'_{s}}} \mid \boldsymbol{x_j^{u'_j}} \in S_{\boldsymbol{a_j\boldsymbol{e_j}}}, \quad \boldsymbol{x_j^{u'_j}} \succeq \boldsymbol{x_j^{u_j}} \text{ for all } j\}.\]
\end{definition}
\begin{remark}
For a set of monomials $M \subseteq S_{(a_1,\dots,a_s)}$, $M$ is $x_1,\dots,x_i$-lex if and only if for all $\boldsymbol{x^u}\boldsymbol{x^v} \in M$, with $\boldsymbol{x^u} \in S_{(a_1,\dots,a_i,0,\dots,0)}$, $\boldsymbol{x^v} \in S_{(0,\dots,0,a_{i+1},\dots,a_s)}$, we have $L(\boldsymbol{x^u})\boldsymbol{x^v} \subseteq M$.
\end{remark}
For the rest of this section we focus on the case that $M$ is a strongly multistable set of monomials. We begin by generalising \cite{aramova2000bigeneric}*{Theorem 4.8}.
\begin{lemma}\label{ismultilex}
    Suppose the strongly multistable set $M$ is $x_1,\dots,x_{i-1}$-lex for some $2 \leq i \leq s$. Then $M^{x_i \text{lex}}$ is $x_1,\dots,x_i$-lex.
\end{lemma}
\begin{proof}
    The proof is essentially the same as \cite{aramova2000bigeneric}*{Theorem 4.8}. We decompose $M$ as follows:
    \[M = \bigcup_{j=1}^{l_1} \bigcup_{k=1}^{l_2} \boldsymbol{x}^{\boldsymbol{u}_j}M_{j,k}\boldsymbol{x}^{\boldsymbol{v}_{k}}\]
    where $\boldsymbol{x}^{\boldsymbol{u}_j} \in S_{(a_1,\dots,a_{i-1},0,\dots,0)}$, $M_{j,k} \subseteq S_{a_i\boldsymbol{e}_i}$ and $\boldsymbol{x}^{\boldsymbol{v}_{k}} \in S_{(0,\dots,0,a_{i+1},\dots,a_s)}$. We will order the $\boldsymbol{x}^{\boldsymbol{u}_j}$ and $\boldsymbol{x}^{\boldsymbol{v}_{k}}$ such that $\boldsymbol{x}^{\boldsymbol{u}_{j'}} \succ \boldsymbol{x}^{\boldsymbol{u}_{j}}$ if and only if $j' < j$, and similar for the $\boldsymbol{x}^{\boldsymbol{v}_{k}}$. Observe that 
    \[M^{x_i \text{lex}} = \bigcup_{j=1}^{l_1} \bigcup_{k=1}^{l_2} \boldsymbol{x}^{\boldsymbol{u}_j}M_{j,k}^{\text{lex}}\boldsymbol{x}^{\boldsymbol{v}_{k}}.\] An element in $M^{x_i \text{lex}}$ can therefore be written as $\boldsymbol{x}^{\boldsymbol{u}_j}\boldsymbol{x^w}\boldsymbol{x}^{\boldsymbol{v}_{k}}$, with $\boldsymbol{x}^{\boldsymbol{u}_j} \in S_{(a_1,\dots,a_{i-1},0,\dots,0)}$, $\boldsymbol{x^w} \in M_{j,k}^{\text{lex}}$ and $\boldsymbol{x}^{\boldsymbol{v}_{k}} \in S_{(0,\dots,0,a_{i+1},\dots,a_s)}$. We need to check that 
    \[L(\boldsymbol{x}^{\boldsymbol{u}_j}\boldsymbol{x^w})\boldsymbol{x}^{\boldsymbol{v}_{k}} \subseteq M^{x_i \text{lex}}.\] To see that this is true, let $\boldsymbol{x^{u'}} \in S_{(a_1,\dots,a_{i-1},0,\dots,0)}$ and $\boldsymbol{x^{w'}} \in S_{a_i\boldsymbol{e}_i}$ with $\boldsymbol{x^{u'}} \in L(\boldsymbol{x}^{\boldsymbol{u}_j})$ and $\boldsymbol{x^{w'}} \in L(\boldsymbol{x^w})$. We wish to show that $\boldsymbol{x^{u'}}\boldsymbol{x^{w'}}\boldsymbol{x^v} \in M^{x_i \text{lex}}$. Since $M$ is $x_1,\dots,x_{i-1}$-lex it follows that for any $\boldsymbol{x^{u'}} \in L(\boldsymbol{x}^{\boldsymbol{u}_j})$ and any $\boldsymbol{x^{w''}} \in M_{j,k}$, $\boldsymbol{x^{u'}}\boldsymbol{x^{w''}}\boldsymbol{x}^{\boldsymbol{v}_{k}} \in M$. In particular $\boldsymbol{x^{u'}} = \boldsymbol{x}^{\boldsymbol{u}_{j'}}$ for some $j' < j$ and $\boldsymbol{x^{w''}} \in M_{j',k}$. Since this holds for any $\boldsymbol{x^{w''}} \in M_{j,k}$, $M_{j,k} \subseteq M_{j',k}$ and consequently $M_{j,k}^{\text{lex}} \subseteq M_{j',k}^{\text{lex}}$. In particular, $\boldsymbol{x^w} \in M_{j',k}^{\text{lex}}$ and so $\boldsymbol{x^{u'}}\boldsymbol{x^w}\boldsymbol{x^v} \in M^{x_i \text{lex}}$. By definition $M^{x_i \text{lex}}$ is $x_i$-lex so it then follows that $\boldsymbol{x^{u'}}\boldsymbol{x^{w'}}\boldsymbol{x^v} \in M^{x_i \text{lex}}$.
\end{proof}

To define $M^{\text{multilex}} = M^{x_1 \dots x_s \text{lex}}$ we now need only check that if $M$ is strongly multistable and $x_1, \dots, x_{i-1}$-lex, then $M^{x_i \text{lex}}$ is also strongly multistable. It then follows that $M^{\text{multilex}}$ is indeed multilex by Lemma \ref{ismultilex}.
\begin{lemma}\label{ismultistable}
    Let $M$ be a strongly multistable set, and let $1 \leq i \leq s$. If $i \geq 2$, suppose further that $M$ is also $x_1,\dots,x_{i-1}$-lex. Then $M^{x_i \text{lex}}$ is strongly multistable.
\end{lemma}
\begin{proof}
We again follow the proof of \cite{aramova2000bigeneric}*{Theorem 4.8}. As in Lemma \ref{ismultilex}, for $i \geq 2$ we decompose
\[M = \bigcup_{j=1}^{l_1} \bigcup_{k=1}^{l_2} \boldsymbol{x}^{\boldsymbol{u}_j}M_{j,k}\boldsymbol{x}^{\boldsymbol{v}_{k}}\]
where $\boldsymbol{x}^{\boldsymbol{u}_j} \in S_{(a_1,\dots,a_{i-1},0,\dots,0)}$, $M_{j,k} \subseteq S_{a_i\boldsymbol{e}_i}$ and $\boldsymbol{x}^{\boldsymbol{v}_{k}} \in S_{(0,\dots,0,a_{i+1},\dots,a_s)}$. If $i=1$, we simply have
\[M = \cup_{k=1}^{l_2} M_{1,k}\boldsymbol{x}^{\boldsymbol{v}_{k}}\] for $M_{1,k} \subseteq S_{a_1\boldsymbol{e_1}}$ and $\boldsymbol{x}^{\boldsymbol{v}_{k}} \in S_{(0,a_2,\dots,a_s)}$. We need to show that for $\boldsymbol{m} \in M^{x_i \text{lex}}$, if $x_{l,p}$ divides $\boldsymbol{m}$, then $x_{l,q}\frac{\boldsymbol{m}}{x_{l,p}} \in M^{x_i \text{lex}}$ for all $0 \leq q < p$. If $l \leq i$ then this holds by Lemma \ref{ismultilex}, since $M^{x_i \text{lex}}$ is $x_1, \dots, x_i$-lex. Otherwise, note that the $X=\{\boldsymbol{x}^{\boldsymbol{v}_1},\dots,\boldsymbol{x}^{\boldsymbol{v}_{l_2}}\}$ form a strongly multistable set since $M$ is strongly multistable. For $l > i$, if $x_{l,p}$ divides $\boldsymbol{x}^{\boldsymbol{v}_{k}}$ for some $k$ then $x_{l,q}\frac{\boldsymbol{x}^{\boldsymbol{v}_{k}}}{x_{l,p}} \in X$ for all $0 \leq q < p$. We write $x_{l,q}\frac{\boldsymbol{x}^{\boldsymbol{v}_{k}}}{x_{l,p}} = \boldsymbol{x}^{\boldsymbol{v}_{k'}}$ for some $k' < k$. Since $M$ is strongly multistable, for fixed $j$, we have $M_{j,k} \subseteq M_{j,k'}$ and consequently $M_{j,k}^{\text{lex}} \subseteq M_{j,k'}^{\text{lex}}$. It follows that if $x_{l,p}$ divides $\boldsymbol{m} \in M^{x_i \text{lex}}$ for $l > i$ then $x_{l,q}\frac{\boldsymbol{m}}{x_{l,p}} \in M^{x_i \text{lex}}$ for all $0 \leq q < p$.
\end{proof}
To obtain the main theorem of this section, we prove one more lemma.
\begin{lemma}\label{multiexist}
Let $I \subseteq S$ be a strongly multistable ideal. Then \[I^{\text{multilex}} = \bigoplus_{(a_1,\dots,a_s) \in \mathbb{N}^s} (I_{(a_1,\dots,a_s)})^{\text{multilex}}\] is an ideal.
\end{lemma}
\begin{proof}
    We now follow the proof of \cite{aramova2000bigeneric}*{Lemma 4.12}. For a set of monomials $M \subseteq S$, define $X_pM = \{x_{p,q}\boldsymbol{m} \mid 0 \leq q \leq n_p, \boldsymbol{m} \in M\}$. To verify that $I^{\text{multilex}}$ is an ideal, we need only check that for any $(a_1,\dots,a_s) \in \mathbb{N}$ and $1 \leq p \leq s$, we have 
    $X_{p} (I_{(a_1,\dots,a_s)})^{\text{multilex}} \subseteq (I_{(a_1,\dots,a_{p}+1,\dots,a_s)})^{\text{multilex}}$. 
    Note that $X_{p}I_{(a_1,\dots,a_s)} \subseteq I_{(a_1,\dots,a_{p}+1,\dots,a_s)}$, hence
    $(X_{p}I_{(a_1,\dots,a_s)})^{\text{multilex}} \subseteq (I_{(a_1,\dots,a_p+1,\dots,a_s)})^{\text{multilex}}$.
    It follows that if we can show that $X_{p}(I_{(a_1,\dots,a_s)})^{\text{multilex}} \subseteq (X_{p}I_{(a_1,\dots,a_s)})^{\text{multilex}}$ then we are done. Decompose $I_{(a_1,\dots,a_s)}$ as 
    \[I_{(a_1,\dots,a_s)} = \bigcup_{j=1}^{l_1} M_j \boldsymbol{x}^{\boldsymbol{v}_j}\] for $M_j \subseteq S_{(a_1,\dots,a_{p-1},0,\dots,0)}$ and $\boldsymbol{x}^{\boldsymbol{v}_j} \in S_{(0,\dots,0,a_p,\dots,a_s)}$. Observe that 
    \[(I_{(a_1,\dots,a_s)})^{x_1 \dots x_{p-1} \text{lex}} = \bigcup_{j=1}^{l_1} M_j^{x_1 \dots x_{p-1} \text{lex}}\boldsymbol{x}^{\boldsymbol{v}_j}.\] Let $X_p\{\boldsymbol{x}^{\boldsymbol{v}_1},\dots,\boldsymbol{x}^{\boldsymbol{v}_{l_1}}\} = \{\boldsymbol{x}^{\boldsymbol{w}_1},\dots,\boldsymbol{x}^{\boldsymbol{w}_{l_2}}\}$. We similarly decompose \[X_{p}I_{(a_1,\dots,a_s)} = \bigcup_{k=1}^{l_2} N_k \boldsymbol{x}^{\boldsymbol{w}_k}\] for $N_k \subseteq S_{(a_1,\dots,a_{p-1},0,\dots,0)}$ and $\boldsymbol{x}^{\boldsymbol{w}_k} \in S_{(0,\dots,0,a_p,\dots,a_s)}$. For all $1 \leq j \leq l_1$ and $0 \leq q \leq n_p$ we have $x_{p,q}\boldsymbol{x}^{\boldsymbol{v}_j} = \boldsymbol{x}^{\boldsymbol{w}_k}$ for some $1 \leq k \leq l_2$, with $M_j \subseteq N_k$. If $M_j \subseteq N_k$ it follows that $M_j^{x_1 \dots x_{p-1} \text{lex}} \subseteq N_k^{x_1 \dots x_{p-1} \text{lex}}$ by Remark \ref{subseteqs}. Consequently, \[X_{p}\bigcup_{j=1}^{l_1} M_j^{x_1 \dots x_{p-1} \text{lex}}\boldsymbol{x}^{\boldsymbol{v}_j} \subseteq \bigcup_{k=1}^{l_2} N_k^{x_1 \dots x_{p-1} \text{lex}}\boldsymbol{x}^{\boldsymbol{w}_k}.\] In other words, $X_p(I_{(a_1,\dots,a_s)})^{x_1 \dots x_{p-1} \text{lex}} \subseteq (X_pI_{(a_1,\dots,a_s)})^{x_1 \dots x_{p-1} \text{lex}}$. A similar argument shows that $X_{p} (I_{(a_1,\dots,a_s)})^{\text{multilex}} \subseteq (X_p (I_{(a_1,\dots,a_s)})^{x_1 \dots x_p \text{lex}})^{x_{p+1} \dots x_s \text{lex}}$. The strategy here is to decompose
    \[(I_{(a_1,\dots,a_s)})^{x_1 \dots x_p \text{lex}}= \bigcup_{j=1}^{l_1} M_j \boldsymbol{x}^{\boldsymbol{v}_j}\] with $M_j \subseteq S_{(0,\dots,0,a_{p+1},\dots,a_s)}$ and $\boldsymbol{x}^{\boldsymbol{v}_{k}} \in S_{(a_1,\dots,a_p,0,\dots,0)}$, and observe that $ (I_{(a_1,\dots,a_s)})^{\text{multilex}} = \bigcup_{j=1}^{l_1}M_j^{x_{p+1} \dots x_s \text{lex}}\boldsymbol{x}^{\boldsymbol{v}_j}$. We then consider the decomposition
    \[X_p(I_{(a_1,\dots,a_s)})^{x_1 \dots x_p \text{lex}} = \bigcup_{k=1}^{l_2} N_k \boldsymbol{x}^{\boldsymbol{w}_k},\] with $X_p\{\boldsymbol{x}^{\boldsymbol{v}_1}, \dots, \boldsymbol{x}^{\boldsymbol{v}_{l_1}}\}=\{\boldsymbol{x}^{\boldsymbol{w}_1},\dots, \boldsymbol{x}^{\boldsymbol{w}_{l_2}}\}$ and $N_k \subseteq S_{(0,\dots,0,a_{p+1},\dots,a_s)}$. As before, we see that for all $1 \leq j \leq l_1$ and $0 \leq q \leq n_p$, we have $x_{p,q} \boldsymbol{x}^{\boldsymbol{v}_j}= \boldsymbol{x}^{\boldsymbol{w}_k}$ for some $1 \leq k \leq l_2$, and consequently $M_j \subseteq N_k$ for this value of $k$. By Remark \ref{subseteqs} we then have $M_j^{x_{p+1} \dots x_s \text{lex}} \subseteq N_k^{x_{p+1} \dots x_s \text{lex}}$. It follows that 
     \[X_{p}\bigcup_{j=1}^{l_1} M_j^{x_{p+1} \dots x_s \text{lex}}\boldsymbol{x}^{\boldsymbol{v}_j} \subseteq \bigcup_{k=1}^{l_2} N_k^{x_{p+1} \dots x_s \text{lex}}\boldsymbol{x}^{\boldsymbol{w}_k},\] or in other words, $X_p(I_{(a_1,\dots,a_s)})^{\text{multilex}} \subseteq (X_p (I_{(a_1,\dots,a_s)})^{x_1 \dots x_p \text{lex}})^{x_{p+1} \dots x_s \text{lex}}$.
     
     We have a chain of inclusions \[X_p (I_{(a_1,\dots,a_s)})^{x_1 \dots x_p \text{lex}} \subseteq (X_p (I_{(a_1,\dots,a_s)})^{x_1 \dots x_{p-1} \text{lex}})^{x_p \text{lex}} \subseteq (X_pI_{(a_1,\dots,a_s)})^{x_1 \dots x_{p} \text{lex}}.\] The second inclusion follows immediately from the previous part and Remark \ref{subseteqs}. The first inclusion is a direct generalisation of the first part of the proof of \cite{aramova2000bigeneric}*{Lemma 4.12}. Combining everything, we obtain
        \[X_p(I_{(a_1,\dots,a_s)})^{\text{multilex}} \subseteq (X_p (I_{(a_1,\dots,a_s)})^{x_1 \dots x_p \text{lex}})^{x_{p+1} \dots x_s \text{lex}} \subseteq (X_pI_{(a_1,\dots,a_s)})^{\text{multilex}}.\] This completes the proof.
\end{proof}
\begin{theorem}\label{multilex}
    For every ideal $I \subseteq S$ which is homogeneous with respect to the $\mathbb{Z}^s$-grading, there exists a multilex ideal with the same Hilbert function.
\end{theorem}
\begin{proof}
As highlighted in Section \ref{twotwo}, for every $I$ there exists a strongly multistable ideal with the same Hilbert function. The ideal $\multigin(I)$ is one such example. Lemmas \ref{ismultilex}, \ref{ismultistable} and \ref{multiexist} then show that for any strongly multistable ideal there exists a multilex ideal with the same Hilbert function.
\end{proof}
We will exploit Theorem \ref{multilex} in Section \ref{twohalf} to find the degrees in which persistence occurs.

\section{Persistence-type results for products of projective spaces}\label{twohalf}
We now establish some preliminary results required to prove Theorem \ref{maintheorem1}. We begin by understanding the structure of Hilbert polynomials on $S$. We then relate multilex ideals and bilex ideals, and use this to generalise a theorem of Crona \cite{crona2006standard}*{Theorem 4.10}. We conclude with the proof of Theorem \ref{maintheorem1} and a comparison to the original result of Gotzmann.
\subsection{The structure of Hilbert polynomials on $S$}
In this subsection we will apply the results of Section \ref{two} on strongly multistable and multilex ideals to better understand the structure of $P_I(t_1,\dots,t_s)$. We begin with a lemma which justifies that the sum of Hilbert polynomials is itself a Hilbert polynomial.
\begin{lemma}\label{addhilb}
    Let $I_1, \dots, I_l \subseteq S$ be monomial ideals, homogeneous with respect to the $\mathbb{Z}^s$-grading and generated in degrees $\leq (a_1, \dots, a_s)$, with each $a_i \geq 2$. Then for $b_i \geq 1$, \[H_{I_1}(b_1,\dots,b_s) + \dots + H_{I_l}(b_1,\dots,b_s) = H_J(b_1,\dots,b_s)\] for a monomial ideal $J$ generated in degrees $\leq(a_1, \dots, a_s)$ in $\cox(\mathbb{P}^{l(n_1+1)-1} \times \dots \times \mathbb{P}^{l(n_s+1)-1})$. It follows that
    \[P_{J}(t_1,\dots,t_s) = P_{I_1}(t_1,\dots,t_s) + \dots + P_{I_l}(t_1,\dots,t_s).\]
\end{lemma}
\begin{proof}
    We begin by relabelling the variables of $I_2$, replacing $x_{i,j}$ with $y_{i,j}$, to obtain an ideal $I_2' \subseteq k[y_{1,0},\dots,y_{1,n_1},\dots,y_{s,0},\dots,y_{s,n_s}]$. Let \[S'=k[x_{1,0}, \dots, x_{1,n_1}, y_{1,0},\dots,y_{1,n_1},\dots,x_{s,0}, \dots, x_{s,n_s},y_{s,0},\dots,y_{s,n_s}],\] with $\deg(x_{i,j})=\deg(y_{i,j})=\boldsymbol{e}_i$. Note that $S'$ is the Cox ring of \[\mathbb{P}^{2n_1 + 1} \times \dots \times \mathbb{P}^{2n_s + 1}.\] Consider the ideal $J_1 = I_1+I_2'+K \subseteq S'$, where $K$ is generated by all monomials of the form $x_{i,j} y_{p,q}$. The Hilbert function of this ideal agrees with $H_{I_1}(b_1,\dots,b_s) + H_{I_2}(b_1,\dots,b_s)$ for $b_i \geq 1$. Note that each $x_{i,j}y_{p,q}$ has degree $\leq(2, \dots, 2)$, so $J_{1}$ is generated in degrees $\leq(a_1, \dots, a_s)$. We also observe that $J_{1}$ is a monomial ideal. Now applying the same argument to $J_1$ and $I_3$ allows us to obtain a monomial ideal $J_2$ in $\cox(\mathbb{P}^{3(n_1+1)-1} \times \dots \times \mathbb{P}^{3(n_s+1)-1})$, generated in degrees $\leq (a_1,\dots,a_s)$, and with $H_{J_1}(b_1,\dots,b_s)+H_{I_3}(b_1,\dots,b_s)=H_{J_2}(b_1,\dots,b_s)$. We continue to repeat this argument to obtain the ideal $J$, which is a monomial ideal in $\cox(\mathbb{P}^{l(n_1+1)-1} \times \dots \times \mathbb{P}^{l(n_s+1)-1})$.
\end{proof}
For $P \in \mathbb{Q}[t_1,\dots,t_s]$ set $\maxdeg(P)$ to be the vector $\boldsymbol{v}=(v_i) \in \mathbb{N}^s$, where each $v_i$ is the maximum power of $t_i$ dividing at least one of the terms of $P$. Combining Lemma \ref{standec} and Lemma \ref{addhilb} we prove the following lemma, which allows us to better understand the Hilbert polynomial of a multigraded ideal.

\begin{lemma}\label{standecomp}
    Let $P_I(t_1,\dots,t_s)$ be the Hilbert polynomial of an ideal $I \subseteq S$, which is homogeneous with respect to the $\mathbb{Z}^{s}$-grading. Fix $(a_1,\dots,a_s) \in \mathbb{N}^s$ such that $a_i \geq 2$ for all $i$ and such that $I^{\text{multilex}}$ is generated in degrees $\leq(a_1,\dots,a_s)$. Suppose that $P_I(t_1,\dots,t_s)$ has $\maxdeg(P_I)=(p_1,\dots,p_s)$. Then we may write
    \[P_I(t_1,\dots,t_s)=\sum_{i_s=0}^{p_s} \dots \sum_{i_2=0}^{p_2} F_{i_2 \dots i_s}(t_1)\binom{t_2-a_2+i_2}{i_2}\dots\binom{t_s-a_s+i_s}{i_s},\] where each $F_{i_2 \dots i_s}(t_1)$ is the Hilbert polynomial for some standard-graded monomial ideal $J_{i_2 \dots i_s}$ generated in degrees $\leq a_1$. Further, we have $H_{J_{i_2 \dots i_s}}(b_1)=F_{i_2 \dots i_s}(b_1)$ for $b_1 \geq a_1$, and consequently $H_I(b_1,\dots,b_r)=P_I(b_1,\dots,b_r)$ when all $b_i \geq a_i$.
\end{lemma}
\begin{proof}
We pass to $J=I^{\text{multilex}}$, which has the same Hilbert function as $I$ but is strongly multistable. Consider the set \[\mathcal{I} = \{(i_2,\dots,i_s) \in \mathbb{N}^{s-1} \mid 0 \leq i_j \leq p_j\}.\] For $\boldsymbol{i}=(i_2,\dots,i_s) \in \mathcal{I}$ consider the set of monomials \[L^{\boldsymbol{i}} = \{\boldsymbol{x^u} \in S_{(0,a_2,\dots,a_s)} \mid i_j = n_j - m^j(\boldsymbol{x^u}) \text{ for all } 2 \leq j \leq s\}.\] Fix some $\boldsymbol{x}^{\boldsymbol{u}_{1}} \in L^{\boldsymbol{i}}$. We consider the set of monomials $M_1^{\boldsymbol{i}} \subseteq S_{(a_1,0,\dots,0)}$ with $\boldsymbol{x}^{\boldsymbol{u}_{1}}M_1^{\boldsymbol{i}} \subseteq J_{(a_1,\dots,a_s)}$. The set $M_1^{\boldsymbol{i}}$ generates a strongly stable ideal $J_{1}^{\boldsymbol{i}}$ in the standard-graded polynomial ring $k[x_{1,0},\dots,x_{1,n_1}]$ with Hilbert function $H_{J_{1}^{\boldsymbol{i}}}(b_1)$. Since $J_{1}^{\boldsymbol{i}}$ is strongly stable and generated in degrees $\leq a_1$ we observe that $H_{J_{1}^{\boldsymbol{i}}}(b_1)=P_{J_{1}^{\boldsymbol{i}}}(b_1)$ for $b_1 \geq a_1$. Consider the set of monomials \[N_1^{\boldsymbol{i}} = S_{(a_1,0,\dots,0)} \setminus M_1^{\boldsymbol{i}}.\] For each $\boldsymbol{x}^{\boldsymbol{v}_j} \in N_1^{\boldsymbol{i}}$, $\boldsymbol{x}^{\boldsymbol{u}_{1}} \boldsymbol{x}^{\boldsymbol{v}_j}$ corresponds to a monomial in $S_{(a_1,\dots,a_s)} \setminus J_{(a_1,\dots,a_s)}$.  For such a monomial $\boldsymbol{x^u} = \boldsymbol{x}^{\boldsymbol{u}_{1}}\boldsymbol{x}^{\boldsymbol{v}_j}$ the corresponding term in the partial Stanley decomposition of $J$ is 
\begin{equation*}
    k[x_{1,m_1(\boldsymbol{x}^{\boldsymbol{v}_j})}, \dots, x_{1,n_1}, x_{2,m_2(\boldsymbol{x}^{\boldsymbol{u}_{1}})}, \dots, x_{2,n_2}, \dots, x_{s,m_s(\boldsymbol{x}^{\boldsymbol{u}_{1}})}, \dots, x_{s,n_s}](-a_1,\dots,-a_s).
\end{equation*}
This contributes 
\begin{equation}\label{corrhilb1}
    \dimn_k(k[x_{1,m_1(\boldsymbol{x}^{\boldsymbol{v}_j})}, \dots, x_{1,n_1}]_{t_1-a_1}) \binom{t_2-a_2+i_2}{i_2}\dots \binom{t_s-a_s+i_s}{i_s}
\end{equation}
to the Hilbert polynomial of $I$, where $i_j = n_j - m^j(\boldsymbol{x}^{\boldsymbol{u}_{1}})$. Since the monomials in $N_1^{\boldsymbol{i}}$ form a partial Stanley decomposition for $J_1^{\boldsymbol{i}}$ we have that 
\[\sum_{\boldsymbol{x}^{\boldsymbol{v}_j} \in N_1^{\boldsymbol{i}}} \dimn_k(k[x_{1,m_1(\boldsymbol{x}^{\boldsymbol{v}_j})}, \dots, x_{1,n_1}]_{t_1-a_1}) = P_{J_1^{\boldsymbol{i}}}(t_1) \]
for $t_1 \geq a_1$. Consequently, the expression \eqref{corrhilb1} becomes
\begin{equation}\label{corrhilb2}
        P_{J_{1}^{\boldsymbol{i}}}(t_1)\binom{t_2-a_2+i_2}{i_2}\dots\binom{t_s-a_s+i_s}{i_s}.
\end{equation}
We repeat the above procedure for all possible monomials $\boldsymbol{x}^{\boldsymbol{u}_l} \in L^{\boldsymbol{i}}$. Note that all monomials in $L^{\boldsymbol{i}}$ for a fixed $\boldsymbol{i}$ will have the same binomial coefficients appearing in the corresponding piece of the Hilbert polynomial \eqref{corrhilb2}. It therefore makes sense for us to add the $P_{J^{\boldsymbol{i}}_l}(t_1)$ together. By Lemma \ref{addhilb}, the sum of the Hilbert polynomials of the $J_{l}^{\boldsymbol{i}}$ agrees with the Hilbert polynomial for a new monomial ideal $J_{i_2 \dots i_s}$ generated in degrees $\leq a_1$. We will denote the Hilbert polynomial of $J_{i_2 \dots i_s}$ by $F_{i_2 \dots i_s}(t_1)$. We have
\begin{equation}\label{idek}
    F_{i_2 \dots i_s}(t_1) = \sum_{\boldsymbol{x}^{\boldsymbol{u}_l} \in L^{\boldsymbol{i}}} P_{J_l^{\boldsymbol{i}}}(t_1).
\end{equation}
Note that since $H_{J^{\boldsymbol{i}}_{l}}(b_1)=P_{J^{\boldsymbol{i}}_{l}}(b_1)$ for all $b_1 \geq a_1$, we have 
\[F_{i_2 \dots i_s}(b_1) = \sum_{\boldsymbol{x}^{\boldsymbol{u}_l} \in L^{\boldsymbol{i}}} P_{J_l^{\boldsymbol{i}}}(b_1) = \sum_{\boldsymbol{x}^{\boldsymbol{u}_l} \in L^{\boldsymbol{i}}} H_{J_l^{\boldsymbol{i}}}(b_1) = H_{J_{i_2 \dots i_s}}(b_1) \] for all $b_1 \geq a_1$.
By equation \eqref{idek} the monomials in $L^{\boldsymbol{i}}$ contribute in total 
\[F_{i_2 \dots i_s}(t_1) \binom{t_2-a_2+i_2}{i_2}\dots\binom{t_s-a_s+i_s}{i_s}\]
to the Hilbert polynomial of $I$. Varying $\boldsymbol{i} \in \mathcal{I}$ and repeating this procedure we obtain \[P_{I}(t_1,\dots,t_s)= \sum_{i_s=0}^{p_s} \dots \sum_{i_2=0}^{p_2} F_{i_2 \dots i_s}(t_1)\binom{t_2-a_2+i_2}{i_2}\dots \binom{t_s-a_s+i_s}{i_s},\] where each $F_{i_2 \dots i_s}(t_1)$ is the Hilbert polynomial of a monomial ideal $J_{i_2 \dots i_s}$ generated in degrees $\leq a_1$. 
\end{proof}
\begin{exmp}\label{exmp2}
    Let \[S = k[x_{1,0},x_{1,1},x_{2,0},x_{2,1},x_{2,2},x_{3,0},x_{3,1}]\]
    be the Cox ring of $ \mathbb{P}^1 \times \mathbb{P}^2 \times \mathbb{P}^1$. Let $J = (x_{1,0},x_{2,0},x_{2,1}) \subseteq S$, and observe that $J$ is a multilex ideal generated in degrees $\leq(2,2,2)$. We have $P_J(t_1,t_2,t_3)=t_3+1$, with $\maxdeg(P_J)=(0,0,1)$. By Lemma \ref{standecomp} we can write
\[P_I(t_1,t_2,t_3) = F_{00}(t_1) + F_{01}(t_1) \binom{t_3-2+1}{1},\] for some standard-graded Hilbert polynomials $F_{00}(t_1)$ and $F_{01}(t_1)$. It follows that we must have $F_{01}(t_1)=1$, $F_{00}(t_1)=2$.
\end{exmp}
The following terminology will be used throughout the rest of this section.
\begin{definition}\label{partial}
    Let $P_{I}(t_1,\dots,t_s)$ be a Hilbert polynomial of an ideal $I \subseteq S$ with $I^{\text{multilex}}$ generated in degrees $\leq (a_1,\dots,a_s)$. By Lemma \ref{standecomp} we can write
    \[P_I(t_1,\dots,t_s)=\sum_{i_s=0}^{p_s} \dots \sum_{i_2=0}^{p_2} F_{i_2 \dots i_s}(t_1)\binom{t_2-a_2+i_2}{i_2}\dots\binom{t_s-a_s+i_s}{i_s},\] where the $F_{i_2 \dots i_s}(t_1) \in \mathbb{Q}[t_1]$ are Hilbert polynomials of standard-graded ideals generated in degrees $\leq(a_1,\dots,a_s)$. We will extend the definition of the $F_{i_2 \dots i_s}(t_1)$, setting $F_{i_2 \dots i_s}(t_1) = 0$ for any $(i_2,\dots,i_s) \in \mathbb{N}^{s-1}$ with some $i_j > p_j$.
    For $1 < r < s$ and $\boldsymbol{b}=(b_1,\dots,b_{r-1},b_{r+1},\dots,b_{s}) \in \mathbb{N}^{s-1}$, define
    \[P_{\boldsymbol{b}}^r(t_r) = \sum_{i_r=0}^{p_r} \dots \sum_{i_2=0}^{p_2} F_{i_2 \dots i_{r} b_{r+1} \dots b_{s}}(b_1)\prod_{j=2}^{r-1}\binom{b_j-a_j+i_j}{i_j}\binom{t_r-a_r+i_r}{i_r} \in \mathbb{Q}[t_r].\] 
    If $r=1$ then we write $\boldsymbol{b}=(b_2,\dots,b_s) \in \mathbb{N}^{s-1}$ and define \[P^{1}_{\boldsymbol{b}}(t_1)= F_{b_2 \dots b_{s}}(t_1).\] For $r=s$ write $\boldsymbol{b}=(b_1,\dots,b_{s-1})$ and define \[P^{s}_{\boldsymbol{b}}(t_s) = P_I(b_1,\dots,b_{s-1},t_s).\]  Notice that if $b_i > p_{i}$ for some $i > r$ then $P^r_{\boldsymbol{b}}(t_r)=0$. For $1 \leq r < s$, $P_{\boldsymbol{b}}^{r}(t_r)$ is the coefficient of $\binom{t_{r+1}-a_{r+1}+b_{r+1}}{b_{r}} \dots \binom{t_s-a_s+b_{s}}{b_{s}}$ in $P_I(b_1,\dots,b_{r-1},t_r, \dots, t_s)$.
    
    For $r > 1$, fix $\boldsymbol{b}=(b_1,\dots,b_{r-1},b_{r+1},\dots,b_s) \in \mathbb{N}^{s-1}$ and $i_r \in \mathbb{N}$. Further, set $\boldsymbol{\hat{b}}=(b_1,\dots,b_{r-2},i_r,b_{r+1},\dots,b_s) \in \mathbb{N}^{s-1}$. Then the coefficient of $\binom{t_r - a_r + i_r}{i_r}$ in $P_{\boldsymbol{b}}^{r}(t_r)$ is given by $P_{\boldsymbol{\hat{b}}}^{r-1}(b_{r-1})$.
\end{definition}
\begin{remark}\label{degmatch}
Suppose $\maxdeg(P_I)=(p_1,\dots,p_s)$. For fixed $b_1,\dots,b_{r-1}$ with $b_i \geq a_i$ for all $i$, $P_{\boldsymbol{b}}^{r}(t_r)$ has degree $\leq p_r$ for all choices of $b_{r+1},\dots,b_{s}$. Further, we claim that there is at least one choice of $b_{r+1}, \dots, b_{s}$ such that equality is achieved. For $r=1$ this follows from the definition, so we focus on the case $r>1$. Observe that since $\maxdeg(P_I)=(p_1,\dots,p_s)$ there is at least one choice of $i_2,\dots,i_{r-1}$, $i_{r+1},\dots,i_s$ with $0 \leq i_j \leq p_j$ for all $j$ such that ${F}_{i_2 \dots i_{r-1} p_r i_{r+1} \dots i_s}(t_1) \neq 0$. Recall from Lemma \ref{standecomp} that $F_{i_2 \dots i_{r-1} p_r i_{r+1} \dots i_s}(b_1)=H_J(b_1)$ for some standard-graded ideal $J$ when $b_1 \geq a_1$. Thus, if ${F}_{i_2 \dots i_{r-1} p_r i_{r+1} \dots i_s}(b_1)=0$ then Macaulay's bound for the Hilbert function of an ideal tells us that ${F}_{i_2 \dots i_{r-1} p_r i_{r+1} \dots i_s}(u) \leq 0$ for $u \geq b_1$. This forces ${F}_{i_2 \dots i_{r-1} p_r i_{r+1} \dots i_s}(t_1)$ to be the zero polynomial, which is a contradiction. Thus ${F}_{i_2 \dots i_{r-1} p_r i_{r+1} \dots i_s}(b_1) \neq 0$. We may therefore fix $b_{r+1}=i_{r+1},\dots,b_s=i_s$ and observe that for this choice of $i_2,\dots,i_{r-1}$
\[F_{i_2 \dots i_{r-1} p_r b_{r+1} \dots b_s}(b_1)\binom{b_2-a_2+i_2}{i_2}\dots\binom{b_{r-1}-a_{r-1}+i_{r-1}}{i_{r-1}} > 0,\]
since the binomial coefficients are never zero for $b_i \geq a_i$. The coefficient of $\binom{t_r-a_r+p_r}{p_r}$ in $P_{\boldsymbol{b}}^{r}(t_r)$ is given by
\[\sum_{i_{r-1}=0}^{p_{r-1}} \dots \sum_{i_2=0}^{p_2} F_{i_2 \dots i_{r-1} p_r b_{r+1} \dots b_s} \binom{b_2-a_2+i_2}{i_2} \dots \binom{b_{r-1}-a_{r-1}+i_{r-1}}{i_{r-1}}.\]
Since $b_1 \geq a_1$, we have $F_{i_2 \dots i_{r_1} p_r b_{r+1} \dots b_{s}}(b_1) \geq 0$ for all choices of $i_2 \dots i_{r-1}$. As established there is one choice of $i_2, \dots, i_{r-1}$ such that this polynomial is strictly positive. it follows that the coefficient of $\binom{t_r-a_r+p_r}{p_r}$ in $P_{\boldsymbol{b}}^{r}(t_r)$ is strictly positive, and so $P_{\boldsymbol{b}}^{r}(t_r)$  has degree $p_r$ for this choice of $b_{r+1}, \dots, b_{s}$.
\end{remark}
\begin{exmp}\label{exmp3}
    Returning to Example \ref{exmp2}, let $P_{\boldsymbol{b}}^{r}(t_r)$ denote the polynomials as in Definition \ref{partial} for $P_J(t_1,t_2,t_3)=t_3+1$. We compute
    $P^{2}_{(2,0)}(t_2)$ and $P^{2}_{(2,1)}(t_2)$. Recall that
    $F_{01}(t_1) = 1$ and $ F_{00}(t_1)= 2$, and that $\maxdeg(P_J)=(0,0,1)$. We obtain
    \begin{align*}
        P^{2}_{(2,0)}(t_2) & = \sum_{i_2=0}^{0} F_{i_2 0}(2) \binom{t_2-2+i_2}{i_2}
        = F_{00}(2) =2, \\
        P^{2}_{(2,1)}(t_2) &= \sum_{i_2=0}^{0} F_{i_2 1}(2) \binom{t_2-2+i_2}{i_2} = F_{01}(2)=1.
    \end{align*}
\end{exmp}
\subsection{Extension of results on bilex ideals}
We now extend known results about bilex ideals to multilex ideals, using the following lemmas.
\begin{lemma}\label{addbistable}
    Let $I_1, \dots, I_l \subseteq T=k[x_0,\dots,x_n,y_0,\dots,y_m]$ be bilex ideals generated in degrees $\leq(a_1, a_2)$. Fix $\overline{a}_1 \in \mathbb{N}$. Then there exists $\overline{n} \in \mathbb{N}$ and a bilex ideal $J$ in $\cox(\mathbb{P}^{\overline{n}} \times \mathbb{P}^m)$ such that for all $b_2 \geq a_2$ \[H_{I_1}(\overline{a}_1,b_2)+\dots +H_{I_l}(\overline{a}_1,b_2)=H_J(1,b_2).\] Further we may assume $J$ is generated in degrees $\leq (1,a_2)$. We have 
    \[P_{I_1}(\overline{a}_1,t_2)+\dots+P_{I_l}(\overline{a}_1,t_2)=P_J(1,t_2)\] as polynomials in $\mathbb{Q}[t_2]$.
\end{lemma}
\begin{proof}
    Fix $d= \binom{\overline{a}_1+n}{\overline{a}_1}$ and let $\overline{n}=ld-1$. Let $T'=k[z_0,\dots,z_{\overline{n}},y_0,\dots,y_m]$, with $\deg(z_i)=(1,0)$ and $\deg(y_j)=(0,1)$ Let $\{\boldsymbol{x}^{\boldsymbol{u}_{1}},\dots,\boldsymbol{x}^{\boldsymbol{u}_d}\}$ be the set of monomials in $T_{(\overline{a}_1,0)}$, and consider the family of maps of vector spaces
    \begin{align*}
        \varphi_{i} : T_{(\overline{a}_1,0)} & \rightarrow T'_{(1,0)} \\
        \boldsymbol{x}^{\boldsymbol{u}_j} & \mapsto x_{j-1+(i-1)d}
    \end{align*}
    for $1 \leq i \leq l$. Define the set of monomials \[M_{b_2}^{i} = \{\varphi_i(\boldsymbol{x^u})\boldsymbol{y^v} \mid \boldsymbol{x^u} \in T_{(\overline{a}_1,0)}, \boldsymbol{y^v} \in T_{(0,b_2)}, \boldsymbol{x^u}\boldsymbol{y^v} \in I_i \} \subseteq T'\] for $1 \leq i \leq l$, $b_2 \in \mathbb{N}$. Notice that for $i \neq j$, $M_{b_2}^{i}$ is disjoint from $M_{b_2}^{j}$, and that each $M_{b_2}^{i}$ is bilex when viewed as an ideal in the ring $k[z_{(i-1)d}, \dots z_{id-1},y_0,\dots,y_m]$. 
    Let $M_{b_2} = \bigcup_{i=1}^{l} M_{b_2}^i$. 
    Then 
    \begin{enumerate}[\normalfont(i)]
        \item $M_{b_2}$ is bilex for the correct relabelling of the $z_0,\dots,z_{ld-1}$,
        \item $H_{I_1}(\overline{a}_1,b_2) + \dots + H_{I_l}(\overline{a}_1,b_2) = |T'_{(1,b_2)}|-|M_{b_2}|$,
        \item for $J=(M_{a_2}) \subseteq T'$, $|T'_{(1,b_2)}|-|M_{b_2}| = H_J(1,b_2)$ for all $b_2 \geq a_2$.
    \end{enumerate}
    To see (i), note that the monomials $N_{b_2}^{i}=\{\boldsymbol{y^v} \in T_{(0,b_2)} \mid z_i\boldsymbol{y^v} \in M_{b_2}\}$ form a lexsegment for any $z_i \in T'$. It follows that $M_{b_2}$ is $y$-lex as in Definition \ref{xlex}, but may not be $x$-lex. However, we can relabel the $z_i$ so that the $N_{b_2}^i$ are (non strictly) descending  in size and thus obtain a bilex set. Part (ii) follows from the definition of $M_{b_2}$, and the disjointness of the $M_{b_2}^i$. Finally, part (iii) holds since each $I_i$ is generated in degrees $\leq (a_1,a_2)$. The result then follows, since $J$ is a bilex ideal after relabelling variables.
\end{proof}
\begin{lemma}\label{justbistable}
    Let $I \subseteq S$ be an ideal, homogeneous with respect to the $\mathbb{Z}^{s}$-grading, with $I^{\text{multilex}}$ generated in degrees $\leq (a_1,\dots,a_s)$. As in Lemma \ref{standecomp} we write 
    \[P_I(t_1,\dots,t_s)=\sum_{i_s=0}^{p_s} \dots \sum_{i_2=0}^{p_2} F_{i_2 \dots i_s}(t_1)\binom{t_2-a_2+i_2}{i_2}\dots\binom{t_s-a_s+i_s}{i_s},\]
    where each $F_{i_2 \dots i_s}$ is the Hilbert polynomial of a standard graded ideal. Let $P_{\boldsymbol{b}}^{r}(t_r) \in \mathbb{Q}[t_r]$ denote a polynomial as in definition \ref{partial} for $P_I(t_1,\dots,t_s)$. Fix an integer $r$ with $2 \leq r \leq s$ and fix $\boldsymbol{b}=(b_1,\dots,b_{r-1},b_{r+1},\dots,b_s) \in \mathbb{N}^{s-1}$ with $b_i \geq a_i$ for all $i < r$. Then there exists $m \in \mathbb{N}$ and a bilex ideal $K \subseteq \cox(\mathbb{P}^m \times \mathbb{P}^{n_r})$ generated in degrees $\leq(1,a_r)$ with $P_{\boldsymbol{b}}^{r}(t_r)=P_K(1,t_r)$.
\end{lemma}
\begin{proof}
Recall from Definition \ref{partial} that if $b_i > p_{i}$ for any $i > r$ then $P_{\boldsymbol{b}}^r(t_r)=0$, so we focus on the case that $b_i \leq p_{i}$ for $i > r$. We pass from $I$ to $J=I^{\text{multilex}}$, which is generated in degrees $\leq (a_1,\dots,a_s)$. We begin with the case $r=s$. For a given monomial $\boldsymbol{x}^{\boldsymbol{v}_j} \in S_{(b_1,\dots,b_{s-2},0,0)}$ set $M_j \subseteq S_{(0,\dots,0,b_{s-1},a_s)}$ to be the set of monomials such that $M_j \boldsymbol{x}^{\boldsymbol{v}_j} \subseteq J_{(b_1,\dots,b_{s-1},a_s)}$, and denote by $J_j \subseteq k[x_{s-1,0},\dots,x_{s-1,n_{s-1}},x_{s,0},\dots,x_{s,n_s}]$ the bigraded ideal generated by $M_j$. Since $J$ is multilex, $J_{j}$ is bilex, and generated in degrees $\leq(b_{s-1},a_s)$. Further, we have
\[P_{J}(b_1,\dots,b_{s-2},t_{s-1},t_s) = \sum_{\boldsymbol{x}^{\boldsymbol{v}_j} \in S_{(b_1,\dots,b_{s-2},0,0)}} P_{J_{j}}(t_{s-1},t_s).\]
Applying Lemma \ref{addbistable}, we observe that there exists $m \in \mathbb{N}$ such that 
\[P_{J}(b_1,\dots,b_{s-1},t_s) = \sum_{\boldsymbol{x}^{\boldsymbol{v}_j} \in S_{(b_1,\dots,b_{s-2},0,0)}} P_{J_{j}}(b_{s-1},t_s) = P_K(1,t_s)\]
for an ideal $K \subseteq k[x_0,\dots,x_m,y_0,\dots,y_{n_s}]$.
For the case $2 \leq r < s$ the polynomial $P^{r}_{\boldsymbol{b}}(t_r)$ is the coefficient of \[\binom{t_{r+1}-a_{r+1}+b_{r+1}}{b_{r+1}}\dots\binom{t_s-a_s+b_{s}}{b_{s}}\] in $P_I(b_1,\dots,b_{r-1},t_r,\dots,t_s)$. Consider the set of monomials \[M = \{\boldsymbol{x^u} \in S_{(0,\dots,0,a_{r+1},\dots,a_s)} \mid m^j(\boldsymbol{x^{u}})=n_j-b_{j} \text{ for all } r < j \leq s\}.\] A monomial $\boldsymbol{x^u} \in S_{(a_1,\dots,a_s)} \setminus J_{(a_1,\dots,a_s)}$ which satisfies $m^j(\boldsymbol{x^u})=n_j-b_{j}$ for all $r < j \leq s$ will contribute a term to the partial Stanley decomposition whose associated Hilbert polynomial is of the form \[G(t_1,\dots,t_r) \binom{t_{r+1}-a_{r+1}+b_{r+1}}{b_{r+1}}\dots\binom{t_s-a_s+b_{s}}{b_{s}}\] for some $G(t_1,\dots,t_r)$. Therefore, to find the coefficient of \[\binom{t_{r+1}-a_{r+1}+b_{r+1}}{b_{r+1}}\dots\binom{t_s-a_s+b_{s}}{b_{s}}\] in $P_I(t_1,\dots,t_r)$ we focus on monomials whose degree $(0,\dots,0,a_{r+1},\dots,a_s)$ part is in the set $M$. For a given $\boldsymbol{x}^{\boldsymbol{u}_i} \in M$ consider the set of monomials $M_i \subseteq S_{(a_1,\dots,a_{r-1},a_r,0 \dots, 0)}$, such that $\boldsymbol{x}^{\boldsymbol{u}_i}M_i \subseteq J_{(a_1,\dots,a_s)}$. Since $J$ is multilex the set $M_i$ generates a multilex ideal \[J_i \subseteq k[x_{1,0},\dots,x_{1,n_1},\dots,x_{r,0} \dots,x_{r,n_r}],\]
 generated in degrees $\leq(a_1,\dots,a_r)$. Varying $\boldsymbol{x}^{\boldsymbol{u}_i}$ we get a collection of these multilex ideals $J_i$. For a given $\boldsymbol{x}^{\boldsymbol{u}_i}$ let $N_i$ denote the degree $(a_1,\dots,a_r,0,\dots,0)$ monomials not in $M_i$. The monomials in $N_i$ form a partial Stanley decomposition for the strongly multistable ideal $J_i$. Further, the monomials $\bigcup_{\boldsymbol{x}^{\boldsymbol{u}_i} \in M} \boldsymbol{x}^{\boldsymbol{u}_i}N_i$ are exactly the monomials in $S_{(a_1,\dots,a_s)} \setminus J_{(a_1,\dots,a_s)}$ satisfying $m^j(\boldsymbol{x^u})=n_j-b_{j}$ for all $r < j \leq s$.
 The pieces of the partial Stanley decomposition of $I$ associated to the monomials in $\boldsymbol{x}^{\boldsymbol{u}_i}N_i$ contribute \[P_{J_i}(t_1,\dots,t_r)\binom{t_{r+1}-a_{r+1}+b_{r+1}}{b_{r+1}}\dots\binom{t_s-a_s+b_{s}}{b_{s}}\] to the Hilbert polynomial of $I$.
It follows that \[P^{r}_{\boldsymbol{b}}(t_r) = \sum_{\boldsymbol{x}^{\boldsymbol{u}_i}\in M} P_{J_i}(b_1,\dots,b_{r-1},t_r).\] 
For a fixed $J_i$ and fixed monomial $\boldsymbol{x}^{\boldsymbol{v}_j} \in S_{(b_1,\dots,b_{r-2},0,\dots,0)}$ define the bigraded ideal \[J_{i,j} \subseteq k[x_{r-1,0}, \dots, x_{r-1,n_{r-1}},x_{r,0},\dots,x_{r,n_r}] ,\] which is generated by the set of monomials $M_{i,j} \subseteq S_{r_{(0,\dots,0,b_{r-1},a_r,0,\dots,0)}}$ such that $M_{i,j}\boldsymbol{x}^{\boldsymbol{v}_j} \subseteq J_{i_{(b_1,\dots,b_{r-1},a_r)}}$. The proof is now similar to the $r=s$ case. The ideals $J_{i,j}$ are bilex since each $J_i$ is multilex. Further, we have \[P_{J_i}(b_1,\dots,b_{r-2},t_{r-1},t_r) = \sum_{\boldsymbol{x}^{\boldsymbol{v}_j} \in S_{r_{(b_1,\dots,b_{r-2},0,0)}}} P_{J_{i,j}}(t_{r-1},t_r).\]
Now applying Lemma \ref{addbistable}, there exists $m_i \in \mathbb{N}$ such that \[P_{J_{i}}(b_1,\dots,b_{r-1},t_r) = \sum_{\boldsymbol{x}^{\boldsymbol{v}_j} \in S_{r_{(b_1,\dots,b_{r-2},0,0)}}}P_{J_{i,j}}(b_{r-1},t_r)=P_{K_i}(1,t_r)\] for a bilex ideal $K_i \subseteq k[x_0,\dots,x_{m_i},y_0,\dots,y_{n_r}]$, generated in degrees $\leq(1,a_r)$. Again applying Lemma \ref{addbistable}, there exists $m \in \mathbb{N}$ such that
\[P_{\boldsymbol{b}}^{r}(t_r) = \sum_{\boldsymbol{x}^{\boldsymbol{u}_i} \in M} P_{J_i}(b_1,\dots,b_{r-1},t_r) = \sum_{\boldsymbol{x}^{\boldsymbol{u}_i} \in M} P_{K_i}(1,t_r)=P_{K}(1,t_r)\]
for a bilex ideal $K \subseteq k[x_0,\dots,x_m,y_0,\dots,y_{n_r}]$ generated in degrees $\leq (1,a_r)$.
\end{proof}
\begin{remark}\label{justbilex}
Lemma \ref{justbistable} shows that for a multilex ideal $I$ generated in degrees $\leq (a_1,\dots,a_s)$ and a collection $\{b_1,\dots,b_{s-1}\}$ with $b_i \geq a_i$ there is a bilex ideal $K$ generated in degrees $\leq(1,a_s)$ such that \[P_I(b_1,\dots,b_{s-1},t_s) = P_K(1,t_s).\] We may further assume that $K$ lies in $\cox(\mathbb{P}^m \times \mathbb{P}^{n_s})$ for some $m \in \mathbb{N}$. The choice of fixing $t_1, \dots t_{s-1}$ was arbitrary, and a similar result holds for any $P(b_1,\dots,b_{i-1},t_i,b_{i+1},\dots,b_s)$.
\end{remark}
We will need the following result of Crona for bigraded ideals.
\begin{theorem}[\cite{crona2006standard}*{Theorem 4.10}]\label{Gotz}
 Let $T=k[x_0,\dots,x_n,y_0,\dots,y_m]$ be the Cox ring of $\mathbb{P}^n \times \mathbb{P}^m$ with the usual $\mathbb{Z}^2$-grading. Let $I$ be a bigraded homogeneous ideal of $T$. Let 
    \begin{equation}
        H_{I}(b_1,b_2) = \binom{n+b_1}{n}\mathfrak{q}(b_1) + \mathfrak{r}(b_1),
    \end{equation}
    be the Euclidean division of $H_{I}(b_1,b_2)$ by $\binom{n+b_1}{n}$. We introduce the notation 
    \[H_I(b_1,b_2)^{\left<b_1\right>_1} = \binom{b_1+1+n}{n}\mathfrak{q}(b_1) + \mathfrak{r}(b_1)^{\left<b_1\right>}.\]
    Fix $b_2 \in \mathbb{Z}_{>0}$. Suppose there is a strongly bistable ideal with the same Hilbert function as $I$ generated in degrees $\leq (a_1, a_2)$. Then
\begin{enumerate}[\normalfont(i)]
    \item \[H_{I}(b_1+1,b_2) = H_I(b_1,b_2)^{\left<b_1\right>_1}\] for $b_1 \gg 0$.
    \item For all $u \geq a_1$, if \[H_{I}(u+1,b_2) = H_I(u,b_2)^{\left<u\right>_1}\] then \[H_{I}(b_1+1,b_2) = H_I(b_1,b_2)^{\left<b_1\right>_1}\] for all $b_1 \geq u$. 
\end{enumerate}
 Similar results hold for fixed $b_1$ and varying $b_2$. In this case, let
 $H_I(b_1,b_2) = \binom{m+b_2}{m}\mathfrak{q}(b_2) + \mathfrak{r}(b_2)$ be the Euclidean division of $H_I(b_1,b_2)$ by $\binom{m+b_2}{m}$, and set \[H_I(b_1,b_2)^{\left<b_2\right>_2} = \binom{m+b_2+1}{m}\mathfrak{q}(b_2) + \mathfrak{r}(b_2)^{\left<b_2\right>}.\]
\end{theorem}
It is natural to use Remark \ref{justbilex} to extend this result to the product of $s$ projective spaces.
\begin{theorem}\label{cronaformore}
    Let $I \subseteq S$ be a homogeneous ideal, whose Hilbert function agrees with that of some multilex ideal generated in degrees $\leq(a_1, \dots, a_s)$. Fix $i \in \mathbb{N}$ with $1 \leq i \leq s$. For fixed $b_1 \geq a_1, \dots, b_{i-1} \geq a_{i-1}, b_{i+1} \geq a_{i+1}, \dots b_s \geq a_s$ consider the Euclidean division of $H_I(b_1,\dots,b_s)$ by $\binom{n_i+b_i}{n_i}$, which we write
    \[H_I(b_1,\dots,b_s) = \binom{n_i+b_i}{n_i}\mathfrak{q}(b_i) + \mathfrak{r}(b_i).\] As in Theorem \ref{Gotz}, we set 
    \[H_I(b_1,\dots,b_s)^{\left<b_i\right>_i} = \binom{n_i+b_i+1}{n_i}\mathfrak{q}(b_i)+\mathfrak{r}(b_i)^{\left<b_i\right>}.\]
    Then we have 
    \begin{enumerate}[\normalfont(i)]
        \item \[H_I(b_1,\dots,b_{i-1},b_i+1,b_{i+1},\dots,b_s) = H_I(b_1,\dots,b_s)^{\left<b_i\right>_i}\] for $b_i \gg 0$.
        \item For all $u \geq a_i$ if 
        \[H_I(b_1,\dots,b_{i-1},u+1,b_{i+1},\dots,b_s) = H_I(b_1,\dots,b_{i-1},u,b_{i+1},\dots,b_s)^{\left<u\right>_i}\] then
        \[H_I(b_1,\dots,b_{i-1},b_i+1,b_{i+1},\dots,b_s) = H_I(b_1,\dots,b_s)^{\left<b_i\right>_i}\] for all $b_i \geq u$.
    \end{enumerate}
\end{theorem}
\begin{proof}
Recall that \[P_I(b_1,\dots,b_{i-1},u,b_{i+1},\dots,b_s)=H_I(b_1,\dots,b_{i-1},u,b_{i+1},\dots,b_s)\] when $u \geq a_i$ and $b_j \geq a_j$ for all possible $j$. By Remark \ref{justbilex} there exists $m \in \mathbb{N}$ such that $H_I(b_1,\dots,b_{i_1},u,b_{i+1},\dots,b_s) = P_K(1,u)$ for some bilex ideal $K$ in $\cox(\mathbb{P}^m \times \mathbb{P}^{n_i})$, generated in degrees $\leq (1,a_i)$. The result then follows by Theorem \ref{Gotz}.
\end{proof}
\begin{remark}
A similar application of Remark \ref{justbilex} allows us to extend Aramova, Crona and De Negri's generalisation of Macaulay's bound on the Hilbert function of a bigraded ideal~\cite{aramova2000bigeneric}*{Theorem 4.18} to multigraded ideals.
\end{remark}
\subsection{Persistence for products of projective spaces}
We now have all the tools required to prove the main theorem of this section.

\begin{proof}[Proof of Theorem \ref{maintheorem1}]
    Let $J \subseteq S$ be an ideal, homogeneous with respect to the $\mathbb{Z}^s$-grading, and with Hilbert polynomial $P_J(t_1,\dots,t_s)=P(t_1,\dots,t_s)$. By the assumption that $P(t_1,\dots,t_s)$ is a Hilbert polynomial, such $J$ exists. Suppose $I^{\text{multilex}}$ and $J^{\text{multilex}}$ are generated in degrees $\leq(a_1,\dots,a_s)$, with $a_i \geq 2$ for all $i$. Recall that $H_I(b_1,\dots,b_s)=P_I(b_1,\dots,b_s)$ and $H_J(b_1,\dots,b_s)=P_J(b_1,\dots,b_s)$ when all $b_i \geq a_i$. Suppose that $\maxdeg(P)=(p_1,\dots,p_s)$. By Lemma \ref{standecomp}, $P(t_1,\dots,t_s)$ can be written in the form 
    \[P(t_1,\dots,t_s)=\sum_{i_s=0}^{p_s} \dots \sum_{i_2=0}^{p_2} F_{i_2 \dots i_s}(t_1)\binom{t_2-a_2+i_2}{i_2}\dots\binom{t_s-a_s+i_s}{i_s},\] where the $F_{i_2 \dots i_s}(t_1)$ are Hilbert polynomials of standard-graded ideals generated in degrees $\leq a_1$. For a known $P(t_1,\dots,t_s)$ we can explicitly find each of the $F_{i_2 \dots i_s}$. Similarly, we can write
    \[P_I(t_1,\dots,t_s)=\sum_{i_s=0}^{q_s} \dots \sum_{i_2=0}^{q_2} G_{i_2 \dots i_s}(t_1)\binom{t_2-a_2+i_2}{i_2}\dots\binom{t_s-a_s+i_s}{i_s},\] for some $q_2,\dots,q_s \in \mathbb{N}$, where each $G_{i_2, \dots, i_s}(t_1)$ is the Hilbert polynomial of some standard-graded ideal generated in degrees $\leq a_1$. The polynomial $P_I(t_1,\dots,t_s)$ is therefore uniquely determined by $q_2,\dots,q_s$ and the $G_{i_2 \dots i_s}$. Our aim is to show that $q_i=p_i$ for all $2 \leq i \leq s$, and $G_{i_2 \dots i_s}(t_1)=F_{i_2 \dots i_s}(t_1)$ for all $i_2, \dots, i_s$ by using induction.
    
    The first step of the proof is to fix the point $(d_1,\dots,d_s) \in \mathbb{N}^s$. We denote by $P_{\boldsymbol{b}}^{r}(t_r)$ the polynomials as in definition \ref{partial} for $P$ and $Q_{\boldsymbol{b}}^{r}(t_r)$ those for $P_I$. Recall from Lemma \ref{justbistable} that for each $2 \leq r \leq s$ and $\boldsymbol{b} \in \mathbb{N}^{s-1}$ with $b_i \geq a_i$ for all $i<r$, there exists a bilex ideal $K \subseteq k[x_0,\dots,x_m,y_0,\dots,y_{n_r}]$ such that $P_{\boldsymbol{b}}^{r}(t_r)=P_K(1,t_r)$. We may further assume $K$ is generated in degrees $\leq(1,a_r)$. The same holds for the $Q_{\boldsymbol{b}}^{r}(t_r)$. In particular we can apply Theorem \ref{Gotz} to both sets of polynomials. Let 
    \[P_{\boldsymbol{b}}^{r}(u) =  \binom{u+n_r}{n_r}\mathfrak{q}(u)+\mathfrak{r}(u)\] be the Euclidean division of $P_{\boldsymbol{b}}^{r}(u)$ by $\binom{u+n_r}{n_r}$, and define $c(P_{\boldsymbol{b}}^{r}(t_r))$ to be the minimum value of $u \in \mathbb{N}$ such that
    \[P_{\boldsymbol{b}}^{r}(u+1) =  \binom{u+1+n_r}{n_r}\mathfrak{q}(u)+\mathfrak{r}(u)^{\left<u\right>}.\] This minimum must exist by Theorem \ref{Gotz} part $(i)$. We now choose $d_1 \in \mathbb{N}$ with $d_1 \geq a_1$ and such that $d_1$ is larger than the maximal Gotzmann number of the Hilbert polynomials $F_{i_2 \dots i_s}$. Once $d_1$ is fixed, set \[\mathcal{B}_2 = \{(b_1,b_3,\dots,b_{s}) \in \mathbb{N}^{s-1} \mid b_1 \in \{d_1,d_1+1\}, 0 \leq b_i \leq p_{i} \text{ for } 2 < i \leq s\},\] and set $c_2 = \max_{\boldsymbol{b} \in \mathcal{B}_2}\{c(P_{\boldsymbol{b}}^{2}(t_2))\}$. We then set $d_2 = \max(a_2,c_2)$. In this way we recursively define 
    \begin{align*}
        \mathcal{B}_r=\{(b_1,\dots,b_{r-1},b_{r+1},\dots,b_s) \in \mathbb{N}^{s-1} \mid b_i \in \{d_i,d_i+1\} \text{ for } i<r,  \\ 0 \leq b_i \leq p_{i} \text{ for } i > r\},
    \end{align*} $c_r = \max_{\boldsymbol{b} \in \mathcal{B}_{r-1}}\{c(P_{\boldsymbol{b}}^{r}(t_r))\}$, and $d_r=\max(a_r,c_r)$, up to $r=s$. With $(d_1,\dots,d_s)$ now fixed, we proceed with the proof.
    
    We have defined $\mathcal{B}_r$ for $2 \leq r \leq s$. We similarly define \[\mathcal{B}_1 = \{(b_2,\dots,b_s) \in \mathbb{N}^{s-1} \mid 0 \leq b_i \leq p_i \text{ for all } i\}.\]
    We assume as in the statement of the theorem that $H_I(b_1,\dots,b_s)=P(b_1,\dots,b_s)$ for all points in $\mathcal{B}_s$. We will prove the following statements by induction on decreasing $r$ for all $1 \leq r \leq s$.
    \begin{enumerate}[\normalfont(i)]
        \item $Q_{\boldsymbol{b}}^{r}(d_r) = P_{\boldsymbol{b}}^{r}(d_r)$ and $Q_{\boldsymbol{b}}^{r}(d_r+1) = P_{\boldsymbol{b}}^{r}(d_r+1)$
        for all $\boldsymbol{b} \in \mathcal{B}_r$.
        \item $q_i=p_i$ for all possible $i > r$.
    \end{enumerate}
    We begin with the base case of $r=s$. For $(i)$, this case is exactly the assumption of Theorem \ref{maintheorem1}, that $H_I(b_1,\dots,b_s)=P(b_1,\dots,b_s)$ for all points in $\{(b_1,\dots,b_s) \in \mathbb{N}^s \mid b_i \in \{d_i,d_i+1\}\}$. For $(ii)$, the statement is vacuously true.
    
    For the induction step, we will show that for $2 \leq r \leq s$, $(i)$ and $(ii)$ holding for $r$ implies $(i)$ and $(ii)$ hold for $r-1$.
    To see this is true, suppose $(i)$ and $(ii)$ hold for some $2 \leq r \leq s$. Let 
    \[P_{\boldsymbol{b_r}}^{r}(d_r) = \binom{d_r + n_r}{d_r}\mathfrak{q}(d_r) + \mathfrak{r}(d_r)\]
    be the Euclidean division of $P_{\boldsymbol{b_r}}^{r}(d_r)$ by $\binom{d_r+n_r}{d_r}$. We chose $d_r$ such that $d_r \geq c(P_{\boldsymbol{b}}^{\boldsymbol{r}}(t_r))$ for all possible $\boldsymbol{b_r} \in \mathcal{B}_{r}$. It follows that
    \begin{equation}\label{matchy2}
    Q_{\boldsymbol{b_r}}^{\boldsymbol{r}}(d_r+1) = P_{\boldsymbol{b_r}}^{r}(d_r+1) = \binom{d_r+1+n_r}{d_r+1}\mathfrak{q}(d_r)+\mathfrak{r}(d_r)^{\left<d_r\right>}
    \end{equation}
    for all $\boldsymbol{b_r} \in \mathcal{B}_r$ by $(i)$. As observed earlier, Lemma \ref{justbistable} implies that we can apply Crona's Theorem \ref{Gotz} to $Q_{\boldsymbol{b_r}}^{r}(t_r)$. Since equation \eqref{matchy2} holds, the assumption of Theorem \ref{Gotz} part $(ii)$ is satisfied. We therefore see that $Q_{\boldsymbol{b_r}}^{r}(u)=P_{\boldsymbol{b_r}}^{r}(u)$ for all $u \geq d_r$ and all $\boldsymbol{b_r} \in \mathcal{B}_r$, implying that $Q_{\boldsymbol{b}}^{r}(t_r)=P_{\boldsymbol{b}}^{r}(t_r)$ in $\mathbb{Q}[t_r]$ for all $\boldsymbol{b_r} \in \mathcal{B}_r$. Since $q_i=p_i$ for all $i > r$ by $(ii)$, we can apply Remark \ref{degmatch} to conclude that $q_r=p_r$. We then match the coefficients of $\binom{t_r-i_r+a_r}{i_r}$ in $Q_{\boldsymbol{b_r}}^{r}(t_r)$ and $P_{\boldsymbol{b_r}}^{\boldsymbol{r}}(t_r)$ for all $0 \leq i_r \leq p_r$ and all $\boldsymbol{b_r} \in \mathcal{B}_r$ to see that
    \begin{equation*}
    Q_{\boldsymbol{b_{r-1}}}^{r-1}(d_{r-1})=P_{\boldsymbol{b_{r-1}}}^{r-1}(d_{r-1}), \quad Q_{\boldsymbol{b_{r-1}}}^{r-1}(d_{r-1}+1)=P_{\boldsymbol{b_{r-1}}}^{r-1}(d_{r-1}+1),
    \end{equation*}
    for all $\boldsymbol{b_{r-1}} \in \mathcal{B}_{r-1}$. This completes the proof of the induction step.
    
    Since $(i)$ and $(ii)$ hold for $r=1$, it follows that $q_i = p_i$ for all $2 \leq i \leq s$, and that \[G_{i_2 \dots i_s}(d_1)=F_{i_2 \dots i_s}(d_1), \quad G_{i_2 \dots i_s}(d_1+1)=F_{i_2 \dots i_s}(d_1+1)\] for all $(i_2,\dots,i_s) \in \mathcal{B}_1$. Since we chose $d_1$ to be larger then the Gotzmann number of every Hilbert polynomial $F_{i_2 \dots i_s}(t_1)$, we have
    \[G_{i_2 \dots i_s}(d_1+1)=F_{i_2 \dots i_s}(d_1+1)=F_{i_2 \dots i_s}(d_1)^{\left<d_1\right>}=G_{i_2 \dots i_s}(d_1)^{\left<d_1\right>}\] for every $(i_2,\dots,i_s) \in \mathcal{B}_1$. We therefore apply Gotzmann's persistence theorem to each $G_{i_2 \dots i_s}(t_1)$ and conclude that $G_{i_2 \dots i_s}(t_1)=F_{i_2 \dots i_s}(t_1)$ in $\mathbb{Q}[t_1]$ for all $(i_2, \dots, i_s) \in \mathcal{B}_1$. It follows that $P_I=P$.
\end{proof}
We finish this section by verifying the Hilbert polynomial for Example \ref{exmp2}, along with two examples which contrast our results with earlier work in the area.
\begin{exmp}\label{exmp4}
    As in Example \ref{exmp2} let 
    \[S = k[x_{1,0},x_{1,1},x_{2,0},x_{2,1},x_{2,2},x_{3,0},x_{3,1}],\] 
    and consider the ideal $I=(x_{2,0},x_{1,0}x_{2,1},x_{1,0}x_{2,2},x_{1,1}x_{2,1}) \subseteq S$. Let $J$ and $P_J(t_1,t_2,t_3)$ be as in Example \ref{exmp2}. We will apply Theorem \ref{maintheorem1} to show that $P_I(t_1,t_2,t_3) = P_J(t_1,t_2,t_3)$. Recall that $F_{01}(t_1)=1, F_{00}(t_1)=2$. The maximal Gotzmann number of these polynomials is two, so we set $d_1=2$. Calculating the rest of the $P^{2}_{\boldsymbol{b}}(t_2)$ as in Example \ref{exmp3}, we obtain \[P^{2}_{(2,0)}(t_2) = P^{2}_{(3,0)}(t_2)= 2, \quad P_{(2,1)}^{2}(t_2)=P_{(3,1)}^{2}(t_2)=1,\] and we have $c(P^{2}_{(b_1,b_3)}(t_2)) \leq 2$ for these polynomials, so we set $d_2=2$. Finally,
    \begin{align*}
    P_{(2,2)}^{3}(t_3)=P_{(2,3)}^{3}(t_3)=P_{(3,2)}^{3}(t_3)=P_{(3,3)}^{3}(t_3) = t_3+1
    \end{align*}
    with $c(P^{3}_{(b_1,b_2)}(t_2)) \leq 2$ for all $b_1,b_2$, so we set $d_3 = 2$. Thus, verifying that $H_I(b_1,b_2,b_3)=P_J(b_1,b_2,b_3)$ for all $(b_1,b_2,b_3)$ with $2 \leq b_i \leq 3$ guarantees that $P_I(t_1,t_2,t_3)=t_3+1$.
\end{exmp}

\begin{exmp}\label{counter2}
The following example shows that Crona's result alone is not enough to compute the Hilbert polynomial of an ideal. Let $I \subseteq \cox(\mathbb{P}^n \times \mathbb{P}^m)$. Suppose there was a point $(d_1,d_2) \in \mathbb{N}^2$ with $I$ generated in degrees $\leq (d_1,d_2)$, and such that
\begin{enumerate}[\normalfont(i)]
    \item $H_{I}(d_1+1,d_2) = H_I(d_1,d_2)^{\left<d_1\right>_1}$,
    \item $H_{I}(d_1,d_2+1) = H_I(d_1,d_2)^{\left<d_2\right>_2}$,
    \item $H_{I}(d_1+1,d_2+1) = H_I(d_1,d_2+1)^{\left<d_1\right>_1}$,
    \item $H_{I}(d_1+1,d_2+1) = H_I(d_1+1,d_2)^{\left<d_2\right>_2}$.
\end{enumerate}
Then using Theorem \ref{Gotz} we would understand the growth of $H_I(b_1,b_2)$ for all $(b_1,b_2) \in (d_1,d_2) + \mathbb{N}^2$, allowing us to verify $P_I(t_1,t_2)$. However the following example shows that such a point may fail to exist. In fact, a point satisfying (i) and (ii) alone may fail to exist. Let \[I = (y_0^2,x_0y_0,x_0y_1^3) \subseteq k[x_0,x_1,y_0,y_1,y_2].\] The Hilbert polynomial of $I$, as calculated in Macaulay2 using the Correspondence Scrolls package~\cite{CorrespondenceScrollsSource}, is $P_I(t_1,t_2)=3t_1+2t_2+1$. Note that we have $H_I(d_1,d_2) = P_I(d_1,d_2)$ for all $(d_1,d_2) \in (1,3) + \mathbb{N}^2$. For fixed $t_2=d_2$, the Hilbert function $H_I(d_1,d_2)$ satisfies (i) when $\mathfrak{q}(d_1) = 0$ and $H_I(d_1+1,d_2) = H_I(d_1,d_2)^{\left<d_1\right>}$. The latter condition occurs when $d_1 \geq 2d_2+4$, since $2d_2+4$ is the Gotzmann number of the polynomial $P(t_1)=3t_1+2d_2+1$. Similarly, for fixed $t_1=d_1$ the Hilbert function satisfies (ii) exactly when $\mathfrak{q}(d_2)=0$ and $H_I(d_1,d_2+1)=H_I(d_1,d_2)^{\left<d_2\right>}$. The latter condition occurs when $d_2 \geq 3d_1+2$, as this is the Gotzmann number of the polynomial $Q(t_2)=3d_1+2t_2+1$. This means for a point $(d_1,d_2)$ to satisfy (i) and (ii) it must satisfy the inequalities
    \[d_1 \geq 2d_2+4 \geq 2(3d_1+2)+4.\]
    Rearranging, we obtain that $d_1 \leq -\tfrac{8}{5}$ which is clearly not possible for $d_1 \in \mathbb{N}$. Thus, no such point exists.
\end{exmp}

The key difference between Theorem \ref{maintheorem1} and Gotzmann's original persistence result lies in the choice of the point $(d_1,\dots,d_s)$. In  Gotzmann's original result we can choose $d$ to be the Gotzmann number of the prospective Hilbert polynomial $P(t)$. Work of Iarrobino and Kleiman~\cite{MR1735271}*{Appendix C} shows that the Gotzmann number of $P$ is equal to \[\text{inf}\{m \mid m \in \reg(I) \text{ for all } I \subseteq S \text{ with } P_I(t)=P(t)\}.\] Work of Bayer shows that this infimum is achieved as $\min(\reg(L))$ where $L$ is the unique saturated lexicographic ideal with Hilbert polynomial $P$. Thus, for saturated lex ideals $I$ and $J$ and $d \in \reg(I) \cap \reg(J)$, we have $H_I(d+1)=H_I(d)^{\left<d\right>}$ and $H_J(d+1)=H_J(d)^{\left<d\right>}$. It follows that if $H_I(d+1)=H_J(d+1)$ and $H_I(d)=H_J(d)$ then $P_I=P_J$, and in fact $I=J$. In other words, we cannot have different saturated lexicographic ideals whose Hilbert functions agree in degree $d$ and $d+1$. We might hope that a similar result is true for products of projective spaces. However, the following example shows that this does not generalise even to the $\mathbb{P}^n \times \mathbb{P}^m$ case.
\begin{exmp}\label{counter1}
     Let $S=k[x_0,x_1,x_2,y_0,y_1,y_2,y_3]$ be the Cox ring of $X = \mathbb{P}^2 \times \mathbb{P}^3$, where $\deg(x_i)=(1,0)$ and $\deg(y_j)=(0,1)$. Let 
         \begin{gather*}
             I=(x_0,y_0,y_1^2,x_1y_1,x_1y_2^3), \\
             J=(y_0,x_0y_1,x_0y_2,x_1y_1,x_1y_2,y_1^3,y_1^2y_2).
         \end{gather*}
    Note that $I$ and $J$ are generated in degrees $\leq (1,3)$, and are both bilex and $B$-saturated, where $B$ is the irrelevant ideal of $X$. Computing the multigraded regularity of $I$ and $J$ with the Macaulay2~\cite{M2} virtual resolutions package~\cite{almousa2020virtual} we observe $(1,3) \in \reg(I) \cap \reg(J)$, or equivalently $(0,2) \in \reg(S/I) \cap \reg(S/J)$. Consider the set $\mathcal{D}=\{(1,3),(2,3),(1,4),(2,4)\}$. Then $H_{I}(t_1,t_2)=H_{J}(t_1,t_2)$ for all $(t_1,t_2) \in \mathcal{D}$. However, \[H_{I}(3,3) = 16 \neq H_{J}(3,3)=17.\] Since the Hilbert functions of both $I$ and $J$ agree with their Hilbert polynomials in degree $(3,3)$ by corollary 2.15 of~\cite{maclagan2003uniform}, $I$ and $J$ necessarily have different Hilbert polynomials. Indeed, we verify in Macaulay2 using the Correspondence Scrolls package~\cite{CorrespondenceScrollsSource} that $P_I(t_1,t_2)=3t_1+2t_2+1$ and $P_J(t_1,t_2)=\tfrac{1}{2}t_1^2+\tfrac{3}{2}t_1+2t_2+2$. 
\end{exmp}
We therefore have to use a more complicated method to find appropriate $\boldsymbol{d}=(d_1,\dots,d_s)$ such that checking the Hilbert function of our ideal in a finite number of points around $\boldsymbol{d}$ verifies its Hilbert polynomial.

\begin{remark}
An increased understanding of the degrees of the generators of $I^{\text{multilex}}$ would allow us to combine Theorem \ref{maintheorem1} with Maclagan and Smith's result on multigraded regularity to obtain a supportive set as in~\cite{haiman2002multigraded}*{Proposition 3.2}, of size $2^s$ for $P(t_1,\dots,t_s)$.
\end{remark}

\section{Extension to more general toric varieties}\label{three}
\subsection{The Picard rank 2 case}
In this section we see how the results of Section \ref{twohalf} apply to Picard-rank-2 toric varieties. Let $d,s \in \mathbb{Z}$ with $d \geq 2$, $1 \leq s \leq d-1$. Let $\boldsymbol{a}=(a_1,\dots,a_s) \in \mathbb{Z}^s$ with $0 \leq a_1 \leq a_2 \leq \dots \leq a_s$. Work of Kleinschmidt~\cite{kleinschmidt1988classification} defines the smooth projective toric variety $X_d(\boldsymbol{a})$ associated to $d$ and $\boldsymbol{a}$. Its Cox ring is  $R=k[z_1,\dots,z_{d+2}]$, with the grading given by the $2 \times (d+2)$ matrix $A$:
 \[
A=
\begin{blockarray}{cccccccc}
\begin{block}{(cccccccc)}
    -a_1 & \dots & -a_s & 0 & 1 & \dots & 1 & 1\\
    1 & \dots & 1 & 1 & 0 & \dots & 0 & 0 \\
\end{block}
\end{blockarray}\quad.
\]
With this choice of grading the semigroup $\mathcal{K}$ of nef line bundles on $X_d(\boldsymbol{a})$ is identified with $\mathbb{N}^2$. Kleinschmidt further showed that every Picard-rank-2 smooth projective toric variety is isomorphic to $X_d(\boldsymbol{a})$ for some choice of $d$ and $\boldsymbol{a}$. We therefore restrict our focus to varieties with Cox rings in the above form. We can relate the ring $R$ to the Cox ring of the product of two projective spaces as follows.

\begin{lemma}\label{surjhom}
Let $R=k[z_1,\dots,z_{d+2}]$ be the Cox ring of a Picard-rank-2 toric variety, and let $J \subseteq R$ be an ideal homogeneous with respect to the $\mathbb{Z}^2$-grading. Let $n=\dimn_k(R_{(1,0)})-1$ and $m=\dimn_k(R_{(0,1)})-1$, and set $S=\cox(\mathbb{P}^n \times \mathbb{P}^m)$. Then there is a surjective map of $k$-algebras $\psi \colon S \rightarrow \widetilde{R}$, where $\widetilde{R} = \bigoplus_{\boldsymbol{b} \in \mathbb{N}^2}R_{\boldsymbol{b}}$.
\end{lemma}

\begin{proof}
Let $S=k[x_0,\dots,x_n,y_0,\dots,y_m]$, with the grading given by $\deg(x_i)=(1,0)$ and $\deg(y_j)=(0,1)$. Note that $n=d-s$, since $R_{(1,0)}$ has a basis given by $\{z_{s+2},\dots,z_{d+2}\}$. We have chosen $S$ such that $R_{(1,0)}$ and $R_{(0,1)}$ are isomorphic to the vector spaces $S_{(1,0)}$ and $S_{(0,1)}$. We define a map of $k$-algebras $\psi \colon S \rightarrow \widetilde{R}$. Explicitly, set $\psi(x_i)=z_{s+2+i}$ and set $\psi(y_i)=e_i$, where $e_0,\dots,e_m$ is the monomial basis of $R_{(0,1)}$. We now verify that the map $\psi$ is surjective.

Consider a monomial $\boldsymbol{z^u}=z_1^{r_1} \dots z_{d+2}^{r_{d+2}} \in R_{(b_1,b_2)}$ for some $(b_1,b_2) \in \mathbb{N}^2$. We have that 
\begin{align}\label{rst}
    r_{s+2}+\dots+r_{d+2} &= r_1a_1+\dots+r_sa_s+b_1, \\
    r_1+\dots+r_{s+1} &= b_2.
\end{align}
We can rearrange the monomial $\boldsymbol{z^u}$ to exhibit it as the product of $b_1$ elements of $R_{(1,0)}$ and $b_2$ elements of $R_{(0,1)}$ as follows. Consider the multiset \[\mathcal{Z} = \bigcup_{i=s+2}^{d+2} \bigcup_{j=1}^{r_i} \{z_i\}.\] In other words, \[\mathcal{Z}=\{z_{s+2},\dots,z_{s+2},\dots,z_{d+2},\dots,z_{d+2}\},\] where each $z_i$ is repeated $r_i$ times. Equation \eqref{rst} implies that $\mathcal{Z}$ has size $r_1a_1+\dots+r_sa_s+b_1$. We relabel $r_1a_1+\dots+r_sa_s$ of the elements of $\mathcal{Z}$ to variables of the form $z_{i,j,k}$ for $1 \leq i \leq s$, $1 \leq j \leq r_i$, $1 \leq k \leq a_i$. For a fixed $i$ between $1$ and $s$ we have $r_ia_i$ elements lying in the multiset $\mathcal{Z}$ labelled $z_{i,1,1},\dots,z_{i,r_i,a_i}$. Let $f$ denote the product of the remaining $b_1$ elements of $\mathcal{Z}$, which as previously observed all lie in $R_{(1,0)}$. Then
\[\boldsymbol{z^u} = f z_{s+1}^{r_{s+1}}\prod_{i=1}^{s}\prod_{j=1}^{r_i} (z_i \prod_{k=1}^{a_i}z_{i,j,k}). \] Note that each \[z_i \prod_{k=1}^{a_i}z_{i,j,k}\] is an element of $R_{(0,1)}$. This rearrangement of the monomial $\boldsymbol{z^u}$ shows that it is the product of elements in $R_{(1,0)}$ and $R_{(0,1)}$, ensuring that $\psi$ is surjective.
\end{proof}
\begin{remark}\label{surjhomrem}
     Suppose $P(t_1,t_2)$ is the Hilbert polynomial for some homogeneous ideal $J \subseteq R$. Define $\widetilde{J} = \bigoplus_{\boldsymbol{b} \in \mathbb{N}^2}J_{\boldsymbol{b}}$. Composing the surjective map from Lemma \ref{surjhom} with the map $\widetilde{R} \rightarrow \widetilde{R}/\widetilde{J}$ we obtain a surjective map $\varphi \colon S \rightarrow \widetilde{R}/\widetilde{J}$. For $I=\ker(\varphi) \subseteq S$ we have $H_I(b_1,b_2)=H_J(b_1,b_2)$ for $(b_1,b_2) \in \mathbb{N}^2$, and it follows that $P_I(t_1,t_2)=P(t_1,t_2)$, ensuring that $P(t_1,t_2)$ is a Hilbert polynomial on $S$. Thus, it follows from Lemma \ref{surjhom} that if $P(t_1,t_2)$ is a Hilbert polynomial on $R$, then it is also a Hilbert polynomial on $S$.
\end{remark}
\begin{proof}[Proof of Theorem \ref{picrank2}]
As in Remark \ref{surjhomrem}, using the surjective homomorphism of Lemma \ref{surjhom} we may find a homogeneous ideal $I \subseteq S$ with $H_I(b_1,b_2)=H_J(b_1,b_2)$ for all $(b_1,b_2) \in \mathbb{N}^2$. Consequently, we apply Theorem \ref{maintheorem1} to $I \subseteq S$ to find appropriate $(d_1,d_2) \in \mathbb{N}^2$ such that if
\begin{gather*}
        H_J(d_1,d_2)=P(d_1,d_2), \quad H_J(d_1+1,d_2)=P(d_1+1,d_2), \\ H_J(d_1,d_2+1)=P(d_1,d_2+1), \quad H_J(d_1+1,d_2+1)=P(d_1+1,d_2+1), 
\end{gather*}
then $P_I(t_1,t_2)=P(t_1,t_2)$. Since $P_I(t_1,t_2)=P_J(t_1,t_2)$ this then guarantees that $P_J(t_1,t_2)=P(t_1,t_2)$.
\end{proof}
\begin{exmp}\label{Hirz}
Consider the Hirzebruch surface $\mathcal{H}_1$ with Cox ring $R=k[z_0,z_1,z_2,z_3]$, with $\deg(z_0)=\deg(z_2)=(1,0)$, $\deg(z_1)=(-1,1)$, $\deg(z_3)=(0,1)$. Consider the ideal $J=(z_0)$ in $R$. We will find the Hilbert polynomial $P_J(t_1,t_2)$ using Theorem \ref{picrank2}. Observe that $R_{(1,0)} $ has basis $\{z_0,z_2\}$ and $R_{(0,1)}$ has basis $\{z_1z_0,z_1z_2,z_3\}$. We therefore have a surjective homomorphism from $S=k[x_0,x_1,y_0,y_1,y_2]$, the Cox ring of $\mathbb{P}^1 \times \mathbb{P}^2$, to $\widetilde{R}$ with kernel $(x_0y_1-x_1y_0)$. Extending the above map we have a surjection from $S$ to $\widetilde{R}/\widetilde{J}$ with kernel $I=(x_0,x_0y_1-x_1y_0)=(x_0,x_1y_0)$, which is multilex. The Hilbert function of $S/I$ in degree $\boldsymbol{t} \in \mathbb{N}^2$ is equal to that of $R/J$. Let $P(t_1,t_2) = t_2+1$ and suppose we want to verify that $P_J(t_1,t_2)=P(t_1,t_2)$. As in Lemma \ref{standecomp}, we write
\[P(t_1,t_2)=F_{0}(t_1)\binom{t_2-a_2+0}{0}+F_{1}(t_1)\binom{t_2-a_2+1}{1}.\] Since $P(t)$ is the Hilbert polynomial of the multilex ideal $(x_0,y_0)$ we can let $(a_1,a_2)=(2,2)$ to obtain 
\[P(t_1,t_2)=F_{0}(t_1)\binom{t_2-2+0}{0}+F_{1}(t_1)\binom{t_2-2+1}{1}.\] It follows that $F_{1}(t_1)=1$, $F_{0}(t_1)=2$, with maximal Gotzmann number 2. Fixing $d_1=2$, we observe that $c(P_{2}^{2}(t_2))=c(P_{3}^{2}(t_2))=c(t_2+1)=1$, so we can also fix $d_2=2$.
Therefore by Theorem \ref{maintheorem1} it is enough to check $H_J(t_1,t_2)=P(t_1,t_2)$ for \[(t_1,t_2) \in \{(2,2),(2,3),(3,2),(3,3)\}\] to verify that $P_J(t_1,t_2)=P(t_1,t_2)$.
\end{exmp}
\subsection{Higher Picard rank toric varieties}\label{highranksec}
The surjective morphism of rings in Lemma \ref{surjhom} is actually a stronger condition than is needed to extend results about products of projective spaces to more general smooth projective toric varieties. By making use of multigraded Castelnuovo-Mumford regularity, we are able to use similar techniques to extend our persistence theorem to any smooth projective toric variety. For the rest of this section, let $R$ be the Cox ring of a smooth projective toric variety $X$ with Picard rank $s$, with a grading given by fixing an isomorphism of $\pic(X)$ with $\mathbb{Z}^s$. Let $\mathcal{K}$ be the semigroup of nef line bundles on $X$.
\begin{definition}\label{whatispoly}
    For an ideal $J \subseteq R$, homogeneous with respect to the $\mathbb{Z}^s$-grading, the Hilbert function of $J$ is given by
    \begin{align*}
        H_J: \mathbb{Z}^s & \rightarrow \mathbb{N} \\
        \boldsymbol{r}=(r_1,\dots,r_s) & \mapsto \dimn_k((R/J)_{\boldsymbol{r}}).
    \end{align*}
    This function agrees with a polynomial $P_J(t_1,\dots,t_s)$ for all $(r_1,\dots,r_s)$ sufficiently far from the boundary of $\nef(X)$. This polynomial is the Hilbert polynomial of the ideal $J$. The notion of being sufficiently far from the boundary of $\nef(X)$ can be made rigorous using Castelnuovo-Mumford regularity, see \cite{maclagan2003uniform}*{Corollary 2.15}.
\end{definition}
\begin{remark}
    It is important to note that the explicit polynomial $P_J(t_1,\dots,t_s)$ we obtain for an ideal $J \subseteq R$ is dependent on our choice of isomorphism $\pic(X) \cong \mathbb{Z}^s$. For example, in the $\mathbb{P}^n$ case it is standard to choose the isomorphism $\mathcal{O}(d) \mapsto d \in \mathbb{Z}$, which gives the Cox ring the standard grading. However we could just as easily choose the isomorphism $\mathcal{O}(d) \mapsto -d \in \mathbb{Z}$, in which case each variable in the Cox ring of $\mathbb{P}^n$ would have degree $-1$, and the Hilbert polynomial of the Cox ring would be $P(b)=\binom{n-b}{-b}$ instead of the usual $P(b)=\binom{n+b}{b}$. In the case of $\mathbb{P}^n$ there is a standard choice of isomorphism, but for more general smooth projective toric varieties this can fail to be the case.
\end{remark}
Maclagan and Smith \cite{maclagan2004multigraded} give a generalisation of Castelnuovo-Mumford regularity for an $R$-module M or coherent sheaf $\mathcal{F}$ on $X$, denoted $\reg(M)$ and $\reg(\mathcal{F})$. Note that $\reg(M)$ and $\reg(\mathcal{F})$ are sets contained inside $\mathbb{Z}^s \cong \pic(X)$. This is different to the standard-graded case, where we define $\reg(M) = \min \{m \in \mathbb{Z} \mid M \text{ is } m \text{-regular}\}$. In the multigraded case there may not be a clear minimum element so it makes sense to instead consider the set of all $\boldsymbol{m} \in \mathbb{Z}^n$ such that $M$ is $\boldsymbol{m}$-regular.

For the rest of this section fix an element $\boldsymbol{c}_1 \neq \boldsymbol{0} \in \reg(X)$, and choose line bundles $\boldsymbol{c}_2,\dots, \boldsymbol{c}_s \in \mathcal{K}$ such that the collection $\mathcal{C} = \{\boldsymbol{c}_1,\dots,\boldsymbol{c}_s\}$ is a basis for $\mathbb{Q}^s$. Such a collection must exist since $X$ is projective, meaning $\nef(X)$ is an $s$-dimensional cone. Note that the following definition and proof work for any such choice of $\mathcal{C}$.

\begin{definition}\label{coxandf}
Let $S$ be the Cox ring of \[\mathbb{P}^{n_1} \times \dots \times \mathbb{P}^{n_s},\] where $n_i=\dimn_k(R_{\boldsymbol{c}_i})-1$. Explicitly, we write
\[S = k[x_{1,0},\dots,x_{1,n_1},\dots,x_{s,0},\dots,x_{s,n_s}],\] with $\deg(x_{i,j}) = \boldsymbol{e}_i \in \mathbb{Z}^{s}$.
Consider the following automorphism \begin{align*}
f\colon \mathbb{Q}^{s} &\rightarrow \mathbb{Q}^s \\
(b_1,\dots,b_s) &\mapsto b_1\boldsymbol{c}_1+\dots+b_s\boldsymbol{c}_s.
\end{align*}
If we explicitly let $\boldsymbol{c}_i=(c_{i1},\dots,c_{is})$, then we have \[f(b_1,\dots,b_s) = \begin{pmatrix}
    b_1 & \dots & b_s
\end{pmatrix} \begin{pmatrix}
    c_{11} & \dots & c_{1s} \\ \vdots & \vdots & \vdots \\ c_{s1} & \dots & c_{ss}
\end{pmatrix}.\]
We may view $f$ as a map on affine spaces $\mathbb{A}^{s}_{\mathbb{Q}} \rightarrow \mathbb{A}^{s}_{\mathbb{Q}}$. From this perspective, let $f^{\#} \colon \mathbb{Q}[t_1,\dots,t_s] \rightarrow \mathbb{Q}[u_1,\dots,u_s]$ be the induced map on the corresponding coordinate rings. Again, we give this map explicitly as 
\[f^{\#} P(t_1,\dots,t_s) = P (\begin{pmatrix} u_1 & \dots & u_s \end{pmatrix} \begin{pmatrix}
    c_{11} & \dots & c_{1s} \\ \vdots & \vdots & \vdots \\ c_{s1} & \dots & c_{ss}
\end{pmatrix}).\]
In particular since the elements $\boldsymbol{c}_1,\dots,\boldsymbol{c}_s$ are linearly independent over $\mathbb{Q}$, $f$ is surjective, and so $f^{\#}$ is injective.
\end{definition}

\begin{definition}\label{admiss}
We say that a polynomial $P(t_1,\dots,t_s) \in \mathbb{Q}[t_1,\dots,t_s]$ is admissible on $R$ if there is a Hilbert polynomial $Q(u_1,\dots,u_s)$ on $S$ such that $f^{\#} P(u_1,\dots,u_s)=Q(u_1,\dots,u_s)$.
\end{definition}
The following lemma allows us to establish that any Hilbert polynomial on $R$ is admissible. We will make use of the following theorem of Maclagan and Smith.
\begin{theorem}[\cite{maclagan2004multigraded}*{Theorem 1.4}]\label{2surjhom}
    For a simplicial projective toric variety $X$ and coherent sheaf $\mathcal{F}$ on $X$, and for $\boldsymbol{p}$ in $\reg(\mathcal{F})$ and $\boldsymbol{q} \in \mathcal{K}$, there is a surjective morphism 
\begin{equation*}
    H^0(X,\mathcal{F}(\boldsymbol{p})) \otimes H^0(X,\mathcal{O}_X(\boldsymbol{q})) \twoheadrightarrow H^0(X,\mathcal{F}(\boldsymbol{p}+\boldsymbol{q})).
\end{equation*}
\end{theorem}
For the rest of this section set the notation $\mathcal{B} = \{ (b_1,\dots,b_s) \in \mathbb{N}^{s} \mid  b_1 \geq 1\}$.
\begin{lemma}\label{picrankhigh}
Let $J \subseteq R$ be an ideal, homogeneous with respect to the $\mathbb{Z}^s$-grading. Let $S$, $f$ be as in Definition \ref{coxandf}.
Then there exists an ideal $I \subseteq S$ such that there is an isomorphism $\psi_{\boldsymbol{b}} : (S/I)_{\boldsymbol{b}} \rightarrow (R/J)_{f(\boldsymbol{b})}$ for any $\boldsymbol{b}=(b_1,\dots,b_s) \in \mathcal{B}$. 
\end{lemma}
\begin{proof}
Given $\boldsymbol{b} \in \mathcal{B}$, we define a map $\tau_{\boldsymbol{b}} \colon S_{\boldsymbol{b}} \rightarrow R_{f(\boldsymbol{b})}
$. Consider the morphism of $k$-algebras $\tau$ induced by
\begin{align*}
    \tau \colon S & \rightarrow R \\
    x_{i,k} & \mapsto \boldsymbol{e}^i_{k+1},
\end{align*}
where $\boldsymbol{e}^i_{k}$ is the $k$th element of the monomial basis for $R_{\boldsymbol{c_{i}}}$. The map $\tau_{\boldsymbol{b}}$ is defined by restricting $\tau$ to $S_{\boldsymbol{b}}$. Let $\pi \colon R \rightarrow R/J$ be the usual quotient map. To prove the lemma, we need to show that $\tau_{\boldsymbol{b}}$ is surjective. We then postcompose with $\pi |_{f(\boldsymbol{b})} \colon R_{f(\boldsymbol{b})} \rightarrow (R/J)_{f(\boldsymbol{b})}$ to obtain a surjective map to $(R/J)_{f(\boldsymbol{b})}$. To see that $\tau_{\boldsymbol{b}}$ is surjective for $\boldsymbol{b} \in \mathcal{B}$, we apply Theorem \ref{2surjhom} to observe that there is a surjective map
\[H^0(X,\mathcal{O}_X(b_1\boldsymbol{c}_1+\dots+(b_s-1)\boldsymbol{c}_s)) \otimes H^0(X,\mathcal{O}_X(\boldsymbol{c}_s)) \twoheadrightarrow H^0(X,\mathcal{O}_X(f(\boldsymbol{b}))).\]
We then repeatedly apply Theorem \ref{2surjhom} to obtain a surjective map
\begin{equation}\label{tenssurj}
\bigotimes_{i=1}^{s} H^0(X,\mathcal{O}_X(\boldsymbol{c}_i))^{\otimes b_i} \twoheadrightarrow H^0(X,\mathcal{O}_X(f(\boldsymbol{b}))).
\end{equation}
Note that there is a canonical isomorphism $H^0(X,\mathcal{O}_X(f(\boldsymbol{b}))) \cong R_{f(\boldsymbol{b})}$. Further, since $h^0(X,\mathcal{O}_X(\boldsymbol{c}_i))=n_i+1$, we can identify each $H^0(X,\mathcal{O}_X(\boldsymbol{c}_i))$ with $k[x_{i,0},\dots,x_{i,n_i}]_1$. The map \eqref{tenssurj} to $R_{f(\boldsymbol{b})}$ is induced by mapping $x_{i,k} \mapsto \boldsymbol{e}^i_{k+1}$, and then identifying tensor products of such elements with their product in $R$. Since this multiplication in $R$ commutes, this map factors through the vector space obtained by replacing each $H^0(X,\mathcal{O}_X(\boldsymbol{c}_i))^{\otimes b_i}$ with $\text{Sym}^{b_i}(H^0(X,\mathcal{O}_X(\boldsymbol{c}_i)))$, which we can itself identify with $k[x_{i,0},\dots,x_{i,n_i}]_{b_i}$. It follows that the map \eqref{tenssurj} factors through the vector space
\[k[x_{1,0},\dots,x_{1,n_1}]_{b_1} \otimes \dots \otimes k[x_{s,0},\dots,x_{s,n_s}]_{b_s} \cong S_b.\]
We now observe that the map induced from $S_{\boldsymbol{b}}$ to $R_{f(\boldsymbol{b})}$ is exactly the map $\tau_{\boldsymbol{b}}$ defined earlier. We therefore obtain a commutative diagram
\[\begin{tikzcd}[cramped]
	\bigotimes_{i=0}^{s} H^0(X,\mathcal{O}_X(\boldsymbol{c}_i))^{\otimes b_i} && H^0(X,\mathcal{O}_X(f(\boldsymbol{b}))) \\
	\bigotimes_{i=0}^{s} k[x_{i,0},\dots,x_{i,n_i}]_1 && R_{f(\boldsymbol{b})} \\
	 S_{\boldsymbol{b}} \\
	\bigotimes_{i=0}^{s} \text{Sym}^{b_i}(H^0(X,\mathcal{O}_X(\boldsymbol{c}_i)))
	\arrow[leftrightarrow, from=1-1, to=2-1, "\cong"]
	\arrow[two heads, from=2-1, to=3-1]
	\arrow[two heads, from=3-1, to=2-3, "\tau_{\boldsymbol{b}}"]
	\arrow[leftrightarrow, from=1-3, to=2-3, "\cong"]
	\arrow[two heads, from=1-1, to=1-3]
	\arrow[two heads, from=2-1, to=2-3]
	\arrow[leftrightarrow, from=3-1, to=4-1, "\cong"].
\end{tikzcd}\]
In particular since equation \eqref{tenssurj} is a surjection it follows that $\tau_{\boldsymbol{b}}$ is also a surjection. Consequently the induced map  $\widetilde{\tau_{\boldsymbol{b}}} \colon S_{\boldsymbol{b}} \rightarrow (R/J)_{f(\boldsymbol{b})}$ is also a surjection. We define an ideal $I = \bigcup_{\boldsymbol{b} \in \mathbb{N}^{s+1}, b_0 \geq 1} (\text{ker}(\widetilde{\tau_{\boldsymbol{b}}}))$. Observe that for $f \in \text{ker}(\widetilde{\tau_{\boldsymbol{b}}})$, $g \in S$ with $\deg(g)=\boldsymbol{b'}$, we have $fg \in \text{ker}(\widetilde{\tau_{\boldsymbol{b+b'}}})$. It follows that $I_{\boldsymbol{b}} = \text{ker}(\widetilde{\tau_{\boldsymbol{b}}})$. We conclude that we have an isomorphism of vector spaces $(S/I)_{\boldsymbol{b}} \cong (R/J)_{f(\boldsymbol{b})}$ for all $\boldsymbol{b} \in \mathcal{B}$. We will denote these isomorphisms by $\psi_{\boldsymbol{b}}$.
\end{proof}
\begin{remark}\label{picrankhighrmk}
    It follows from Lemma \ref{picrankhigh} that for any homogeneous $J \subseteq R$ there exists a homogeneous ideal $I \subseteq S$ such that $H_I(b_1,\dots,b_s)=H_J(f(b_1,\dots,b_s))$ for all $(b_1,\dots,b_s) \in \mathcal{B}$. It follows that $P_J(t_1,\dots,t_n)$ is admissible, with $f^{\#} P_J(u_1,\dots,u_s)=P_I(u_1,\dots,u_s)$.
\end{remark}
Although the situation is a little more complicated than the Picard-rank-2 case, we are now able to prove Theorem \ref{pichighpers}.
\begin{proof}[Proof of Theorem \ref{pichighpers}]
Let $S$ be as in Definition \ref{coxandf} and consider the ideal $I \subseteq S$ from Lemma \ref{picrankhigh} with $(S/I)_{\boldsymbol{b}} \cong (R/J)_{f(\boldsymbol{b})}$ for all $\boldsymbol{b} \in \mathcal{B}$. As established in Remark \ref{picrankhighrmk}, $P(t_1,\dots,t_s)$ is an admissible polynomial, with $f^{\#} P(u_1,\dots,u_s) = Q(u_1,\dots,u_s)$ for a Hilbert polynomial $Q(u_1,\dots,u_s)$ on $S$. We observe that for $\boldsymbol{b}=(b_1,\dots,b_s) \in \mathcal{B}$ we have $H_I(\boldsymbol{b})=H_J(f(\boldsymbol{b}))$. In particular we therefore have an equality of polynomials \[P_I(u_1,\dots,u_s)=f^{\#}P_J(u_1,\dots,u_s).\] We apply Theorem \ref{maintheorem1} to $I$ to find appropriate $(d_1,\dots,d_{s}) \in \mathcal{B}$ such that verifying 
\begin{equation*}
H_I(b_1,\dots,b_s)=Q(b_1,\dots,b_s)
\end{equation*}
for all $\boldsymbol{b}$ with $b_i \in \{d_{i},d_{i}+1\}$ guarantees that $P_I(u_1,\dots,u_s)=Q(u_1,\dots,u_s)$. Equivalently, verifying
\begin{equation*}\label{hilbpol}
H_J(f(b_1,\dots,b_s))=P(f(b_1,\dots,b_s)) 
\end{equation*}
for all possible $\boldsymbol{b}$ with $b_i \in \{d_{i},d_{i}+1\}$ guarantees that $f^{\#}P_J(u_1,\dots,u_s)=f^{\#}P(u_1,\dots,u_s)$. Thus, checking $H_J(r_1,\dots,r_s)=P(r_1,\dots,r_s)$ at exactly the points \[(r_1,\dots,r_s) \in \mathcal{D} = \{f(\boldsymbol{b}) \in \mathbb{N}^s \mid b_i \in \{d_i,d_i+1\} \quad \forall \quad i\}\] allows us to conclude that $f^{\#}P_J(u_1,\dots,u_s)=f^{\#}P(u_1,\dots,u_s)$. As highlighted in Definition \ref{coxandf}, $f^{\#}$ is injective, and so it follows that $P_J(t_1,\dots,t_s)=P(t_1,\dots,t_s)$.
\end{proof}
\begin{exmp}
    Consider the fan $\Sigma \subseteq \mathbb{R}^3$ with rays 
    \begin{align*}
    \rho_0=(1,0,0), \quad \rho_1=(0,1,0), \quad \rho_2=(-1,1,0), \\
    \rho_3=(0,-1,0), \quad \rho_4=(0,0,1), \quad \rho_5=(0,0,-1),
    \end{align*}
    and maximal cones generated by the subsets
    \begin{align*}
        \{\rho_0,\rho_1,\rho_4\},\{\rho_0,\rho_1,\rho_5\},\{\rho_1,\rho_2,\rho_4\},\{\rho_1,\rho_2,\rho_5\}, \\
        \{\rho_2,\rho_3,\rho_4\},\{\rho_2,\rho_3,\rho_5\},\{\rho_0,\rho_3,\rho_4\},\{\rho_0,\rho_3,\rho_5\}.
    \end{align*}
The associated normal toric variety $X$ is smooth and projective with Picard rank 3. After fixing an isomorphism $\pic(X) \cong \mathbb{Z}^3$ we write the Cox ring of $X$ as $R = k[y_0,\dots,y_5]$, where $\deg(y_0)=\deg(y_2) = (-1,1,0)$, $\deg(y_1)=(1,0,0)$, $\deg(y_3)=(0,1,0)$ and $\deg(y_4) = \deg(y_5)=(0,0,1)$. The nef cone of $X$ is generated by $\boldsymbol{c}_1=(-1,1,0)$, $\boldsymbol{c}_2=(0,1,0)$ and $\boldsymbol{c_3}=(0,0,1)$, and $\reg(R)=\mathcal{K}$, so we may set $\mathcal{C}=\{\boldsymbol{c}_1,\boldsymbol{c}_2,\boldsymbol{c_3}\}$. We have
\begin{align*}
    \quad R_{\boldsymbol{c}_1} = \left< y_0, y_2 \right>, \quad R_{\boldsymbol{c}_2} = \left< y_3, y_0y_1,y_1y_2 \right>, \quad R_{\boldsymbol{c_3}}= \left< y_4,y_5 \right>,
\end{align*}
so we set \[S = k[x_{1,0},x_{1,1},x_{2,0},x_{2,1},x_{2,2},x_{3,0},x_{3,1}],\] which is the Cox ring of $ \mathbb{P}^1 \times \mathbb{P}^2 \times \mathbb{P}^1$. Consider the ideal $J = (y_3,y_0^2y_1,y_0y_1y_2) \subseteq R$. The ideal $I$ in Lemma \ref{picrankhigh} is given by $I=(x_{2,0},x_{1,0}x_{2,1},x_{1,0}x_{2,2},x_{1,1}x_{2,1})$, which is multilex. We have $(S/I)_{\boldsymbol{b}} \cong (R/J)_{f(\boldsymbol{b})}$ for $\boldsymbol{b} \in \mathbb{N}^3$ with $b_1 \geq 1$. Consider the polynomial $P(t_1,t_2,t_3) = t_3+1$, with $Q(u_1,u_2,u_3)=f^{\#}P(u_1,u_2,u_3) = u_3 + 1$. To check that $P_J(t_1,t_2,t_3)=P(t_1,t_2,t_3)$ we need only apply Theorem \ref{maintheorem1} to $I \subseteq S$ to find appropriate $d_1,d_2,d_3 $ such that checking $H_I(b_1,b_2,b_3)=Q(b_1,b_2,b_3)$ for $b_i \in \{d_i,d_i+1\}$ confirms that $P_I=Q$. We now observe that modulo a change of notation this is identical to the case of Example \ref{exmp4}, and set $(d_1,d_2,d_3)=(2,2,2)$. We calculate the values of $f(\boldsymbol{b})$ for $\boldsymbol{b}=(b_1,b_2,b_3) \in \mathbb{N}^3$ with $2 \leq b_i \leq 3$. We obtain eight vertices for our zonotope: 
\begin{align*}
    \mathcal{D} = \{(-2,4,2),(-2,4,3),(-2,5,2),(-2,5,3), \\
    (-3,5,2),(-3,5,3),(-3,6,2),(-3,6,3)\}.
\end{align*}
Checking that $H_I(r_1,r_2,r_3)=P(r_1,r_2,r_3)$ for all $(r_1,r_2,r_3) \in \mathcal{D}$ guarantees that $P_J=P$.
\end{exmp}


\end{document}